\documentclass[11pt]{amsart}
\input{luatex-pdf}
\usepackage{epsfig, amssymb, amsfonts, color, mathrsfs, graphicx}
\usepackage[all]{xy}
\usepackage[margin=1.5in]{geometry}

\usepackage{amscd,amssymb,verbatim}
\usepackage{enumerate}
\usepackage{dtk-logos}
\usepackage{mathrsfs}
\usepackage[T1]{fontenc}

\newtheorem{lemma}{Lemma}[section]
\newtheorem{prop}[lemma]{Proposition}
\newtheorem{thm}[lemma]{Theorem}

\newtheorem{cor}[lemma]{Corollary}
\newtheorem{defn}[lemma]{Definition}
\newtheorem{lem}[lemma]{Lemma}
\newtheorem{rem}[lemma]{Remark}

\newtheorem{axiom}[lemma]{Axiom}
\newtheorem{example}[lemma]{Example}

\newtheorem*{special theorem}{My Specially-Named Theorem}

\newcommand{\N} { {\mathbb N} }
\newcommand{\Z} { {\mathbb Z} }
\newcommand{\Q} { {\mathbb Q} }

\newcommand{\A} { {\mathbb A} }

\newcommand{\C} { {\mathbb C} }

\newcommand{\DM} { {\mathrm{DM}} }
\newcommand{\Sh} { {\mathrm{Sh}} }

\newcommand{\M} { {\mathcal M} }
\newcommand{\Gal} { {\mathrm{Gal}} }

\renewcommand{\P}{\mathbf{P}}

\newcommand{\sA}{\mathcal{A}}

\newcommand{\sT}{\mathcal{T}}

\newcommand{\sX}{\mathcal{X}}
\newcommand{\sB}{\mathcal{B}}

\newcommand{\gm}{{\operatorname{gm}}}
\newcommand{\eff}{{\operatorname{eff}}}
\newcommand{\alg}{{\operatorname{alg}}}

\newcommand{\Spec}{\operatorname{Spec}}

\newcommand{\Cor}{\operatorname{Cor}}

\newcommand{\Coker}{\operatorname{Coker}}

\newcommand{\Hom}{\operatorname{Hom}}

\newcommand{\HI}{\operatorname{HI}}

\newcommand{\Sm}{\operatorname{Sm}}
\newcommand{\Sch}{\operatorname{Sch}}
\renewcommand{\epsilon}{\varepsilon}

\newcommand{\et}{{\operatorname{\acute{e}t}}}
\newcommand{\by}{\xrightarrow}

\newcommand{\skipline}{\vspace{12pt}}

\title{The algebraic part of motivic cohomology}
\date{}
\author{Tohru Kohrita}
\email[Tohru Kohrita]{tohru.kohrita@fu-berlin.de}
\email[Bruno Kahn]{bruno.kahn@imj-prg.fr}

\begin{document}

\maketitle

\hfill With an appendix by Bruno Kahn

\begin{abstract}
Motivated by Murre's work on universal regular homomorphisms on Chow groups in codimension $2,$ we generalize the algebraic equivalence relation and regular homomorphisms to the context of Voevodsky motives over a field. In the Nisnevich topology, we prove the existence of {\em universal} regular homomorphisms for a certain class of motivic cohomology groups, recovering Murre's theorem and the existence of Picard and Albanese varieties as special cases. This class also includes interesting cases such as higher Chow groups and Milnor $K$-groups. The appendix by Kahn proves that, for \'etale motives, universal regular homomorphisms exist for all geometric motives and compares them with those in the Nisnevich topology when both exist.
\end{abstract}


\section{Introduction}\label{section: introduction}
Let $\DM_{Nis}^{eff}(k)$ be the triangulated category of effective Voevodsky motives with integral coefficients over a perfect field $k$ in the sense of \cite[Definition 11.1.1]{Cisinski-Deglise}. It is a closed tensor category and contains the bounded above version $\DM^{eff}_{-}(k)$ of the derived category of motives defined by Voevodsky (\cite{VSF5}; also studied in \cite{MVW}) as a full tensor triangulated subcategory (\cite[Example 11.1.3]{Cisinski-Deglise}). The internal hom in $\DM_{Nis}^{eff}(k)$ is denoted by $\underline{\Hom}.$ Let us review the necessary terminology to explain the content of this paper.

Let $\Sch/k$ be the category of separated schemes of finite type over $k$ and let $(\Sch/k)^{prop}$ be the category with the same objects but only with proper morphisms. $\DM_{Nis}^{eff}(k)$ is equipped with monoidal functors $M\colon \Sch/k\longrightarrow \DM_{Nis}^{eff}(k)$ and $M^c\colon (\Sch/k)^{prop}\longrightarrow \DM_{Nis}^{eff}(k).$ For $X\in \Sch/k,$ $M(X)$ (resp., $M^c(X)$) is by definition the Nisnevich sheaf with transfers $\Z_{tr}(X):=\Cor_k(-,X)$ (resp., $z_{equi}(X,0)$) viewed as a complex concentrated in degree $0.$ Here, $\Cor_k$ stands for the group of finite correspondences in \cite[Definition 1.1]{MVW} and $z_{equi}(X,0)$ is the sheaf defined in \cite[Definition 16.1]{MVW}. Note that $\Z_{tr}(X)$ (resp., $z_{equi}(X,0)$) is indeed a Nisnevich sheaf with transfers and functorial for morphisms in $\Sch/k$ (resp., $(\Sch/k)^{prop}$) (\cite[4.1]{VSF5}; see also \cite[Corollary 3.6.3]{VSF2}). The object $M(X)$ (resp., $M^c(X)$) is called the motive of $X$ (resp., with compact supports). By definition, we have $M(X)=M^c(X)$ if $X$ is proper. The motive $\Z:=M(\Spec k)$ of the base field serves as the monoidal unit. 

For non-negative integers $n,$ $\Z(n)$ denotes the Suslin-Voevodsky motivic complex (\cite[Definition 3.1]{MVW}). It is a complex of Nisnevich sheaves with transfers and hence is an object in $\DM_{Nis}^{eff}(k).$ For $A\in \DM_{Nis}^{eff}(k),$ we write $A(n):=A\otimes \Z(n).$

In this article we define an algebraic equivalence relation on the group $\Gamma(A):=\Hom_{\DM_{Nis}^{eff}(k)}(\Z,A)$ for an arbitrary motive $A\in\DM_{Nis}^{eff}(k)$ (Definition~\ref{defn: algebraic part}). The subgroup $\Gamma_{alg}(A)$ consisting of elements trivial with respect to this equivalence relation is called the algebraic part of $\Gamma(A).$ If $k$ is algebraically closed, $\Gamma_{alg}(A)$ turns out to be divisible (Proposition~\ref{prop: algebraic part is divisible}). This leads us to consider the representability of $\Gamma_{alg}(A)$ by semi-abelian varieties. We generalize Samuel's idea of regular homomorphisms and consider the problem whether universal regular homomorphisms exist; see Definition~\ref{defn: regular homomorphism}.

Before stating the existence theorem, notice that if $A=M(X)$ ($X\in\Sch/k$), the group $\Gamma(A)$ is canonically isomorphic to the Suslin homology $H_0(X,\Z)$ (\cite[Proposition 14.18]{MVW}). In this case, $\Gamma_{alg}(A)$ often agrees with the degree zero part of $H_0(X,\Z)$ (Proposition~\ref{prop: algebraic part of zero cycles}, Remark~\ref{rem: counterexample}). If $A=\underline{\Hom}(M(X),\Z(r)[2r])$ ($X\in\Sm/k, r\in\Z_{\geq0}$), there is a natural isomorphism between $\Gamma(A)$ and the Chow group $CH^r(X)$ (\cite[Theorem 19.1]{MVW}). In this case (Proposition~\ref{prop: comparison with the classical algebraic part}), the algebraic equivalence on $\Gamma(A)$ agrees with Weil-Fulton's algebraic equivalence for Chow groups (\cite{Weil, Fulton}; see Subsection~\ref{section: The case of Chow groups}). 

Universal regular homomorphisms exist in the following case. 

\begin{thm}[{Corollary~\ref{cor: existence in dim zero} and Theorem~\ref{thm: existence main}}]\label{thm: intro 1}
Assume that the base field $k$ is algebraically closed. Let $X\in\Sch/k.$ Let $m$ and $n$ be integers such that $m\leq n+2$ and $n\geq0.$ Then, there exist universal regular homomorphisms for $\Gamma_{alg}(M(X))$ and $\Gamma_{alg}(\underline{\Hom}(M^c(X),\Z(n)[m])).$
\end{thm}

The claim for $\Gamma_{alg}(M(X))$ is a direct consequence of the existence of Serre-Ramachandran's Albanese schemes proved in \cite{Ramachandran}. If $X$ is connected smooth and proper, the case $(m,n)=(2,1)$ is known as the theory of Picard varieties (\cite[Proposition 9.5.10]{Kleiman}), and the case $(m,n)=(4,2)$ is a theorem of Murre (\cite[Theorem A]{Murre}). Just like Murre's proof uses Merkurjev-Suslin's norm residue isomorphism theorem in degree two (\cite{Merkurjev-Suslin}), the proof of Theorem~\ref{thm: intro 1} uses the Beilinson-Lichtenbaum conjecture, which is now a theorem (\cite[Introduction, Theorem C]{HW-norm-residue}) via the \emph{full} norm residue isomorphism theorem (\cite{Haesemeyer-Weibel, Voevodsky mod l, Weibel}). The restriction on the indices $(m,n)$ in the theorem is imposed by the use of this result. The appendix by Kahn proves the existence of universal regular homomorphisms for all {\em \'etale} geometric motives and compares them with those in the Nisnevich topology when both exist.

In the final subsection (Subsection~\ref{subsection: rationality}), we consider the rationality of universal regular homomorphisms. Here we closely follow the argument in \cite[Theorem 4.4]{ACMV}.

It may be worth mentioning one technical point. In the usual theory of algebraic equivalence for smooth proper schemes, it is a key property that two algebraic cycles are algebraically equivalent if and only if they appear in the same family of cycles parametrized by an abelian variety (\cite[Lemma 9]{Weil}). An analogous statement with semi-abelian varieties (Proposition~\ref{prop: parametrization by semiabelian varieties}) plays a similar role in this paper, and this property, as the referee pointed out and outlined the proof, follows via Lemma~\ref{lem: axiom} from Voevodsky's calculation of the motives of smooth \emph{affine} curves.

The referee's contribution is significant in this paper. Some explicitly appear in the text, but others got hidden in the revision process. The contributions of the latter nature include the following. (1) The author originally worked exclusively with motivic cohomology with compact supports. It was the referee who explained how to develop the axiomatic framework as in Section 2, and then apply it to the case of motives. Most importantly, Definition~\ref{defn: algebraic part} and Axiom~\ref{axiom} are due to the referee. These simplified the author's previous arguments, particularly the proofs of Propositions~\ref{prop: algebraic part} and \ref{prop: parametrization by semiabelian varieties}, and unified cases that had been treated separately. (2) The author originally worked only over an algebraically closed base field. It was the referee who suggested working over a perfect (sometimes even non-perfect) field whenever possible and explained how to do it. This turned out fruitful in the study of algebraic part, and we obtained the current version of Propositions~\ref{prop: comparison with the classical algebraic part} and \ref{prop: algebraic part of zero cycles}. The use of the symmetric power construction and \cite[Theorem 3.13]{MilneJV} in the proof of Proposition~\ref{prop: algebraic part of zero cycles} is also due to the referee. Finally, Subsection~\ref{subsection: rationality} would have been nonexistent, had it not been for the referee's explicit question on the topic. The referee also directed the author to the paper \cite{ACMV}. For this subsection, we would also like to thank Charles Vial for pointing out independently that the method in {\em ibid.} is applicable. (3) The definition of unpointed regular homomorphisms is due to the referee (Definition~\ref{defn: unpointed regular homomorphism}). Subsection~\ref{section: unpointed regular homomorphism} should be attributed to him or her. The idea of unpointed regular homomorphisms enables us to state the relation to Ayoub and Barbieri-Viale's work (Remark~\ref{rem: relation to ABV}). It also made the comparison with Serre-Ramachandran's Albanese varieties simpler and conceptual (Subsection~\ref{subsection: Relation with Ramachandran's Albanese schemes}).

\skipline
{\it Convention.} From now on, schemes are assumed separated and of finite type over a base field $k,$ and morphisms of schemes are those over the base field except in  Subsection~\ref{subsection: rationality}. Group schemes are also assumed to be defined over $k.$ 

$\Sch/k$ denotes the category of schemes (separated and of finite type over $k$) and $\Sm/k$ the full subcategory of smooth schemes.

A scheme $X$ pointed at $x$ means a pair of a scheme $X$ and a rational point $x\in X(k).$ A group scheme is considered pointed at the unit $0$ unless 
otherwise noted. A curve means a connected scheme of pure dimension one. 

\skipline
{\it Acknowledgements.} This paper is based on the author's dissertation. The author wishes to thank his advisors Thomas Geisser and Hiroshi Saito for their constant advice and encouragement. This paper is a result of Thomas Geisser's suggestion to consider a motivic version of Murre and H. Saito's work on regular homomorphisms. As mentioned above, the referee's contribution was enormous. We would like to thank him or her for the careful reading and insightful suggestions. In the very first version of this work, Theorem~\ref{thm: intro 1} was stated under resolution of singularities. We thank Shuji Saito for pointing out that the assumption was unnecessary. We also thank Bruno Kahn for helpful discussions and writing an appendix. We would also like to thank Federico Binda, Ryo Horiuchi, Shane Kelly, Amalendu Krishna, Takashi Maruyama, Hiroyasu Miyazaki and Rin Sugiyama for helpful conversations.


\section{Universal regular homomorphisms}\label{chapter: Algebraic representatives}

Although the interest of this paper is motivic cohomology, the formalism of algebraic part and regular homomorphisms can be developed in a categorical setting. The regular homomorphism in this paper is a generalization of what Samuel called a rational homomorphism in the context of Chow groups (\cite[Section 2.5]{Samuel}; see also \cite[Definition 1.6.1]{Murre}). To motivate our discussion, let us review the classical case of Chow groups. 


\subsection{The case of Chow groups}\label{section: The case of Chow groups}

Let $X$ be a scheme in $\Sch/k.$ The quotient $CH^r(X):=Z^r(X)/_{\sim_\mathrm{rat}}$ of the free abelian group $Z^r(X)$ generated by the set of cycles of codimension $r$ on $X$ modulo rational equivalence (\cite[Sections 1.3 and 1.6]{Fulton}) is called the Chow group of $X$ in codimension $r.$ There is another coarser equivalence relation on $Z^r(X)$ called algebraic equivalence (\cite[Definition 10.3]{Fulton}). The cycles algebraically equivalent to zero form a subgroup of $Z^r(X),$ and its image $A^r(X)$ in $CH^r(X)$ is called the algebraic part of $CH^r(X).$ By definition, we have the equality
\[A^r(X)=\bigcup_{\substack{T\in \Sm/k,~{\text{connected}}\\ t_0,t_1\in T(k)}}\mathrm{im}\{CH^r(T\times X)\buildrel (\cdot)_{t_1}-(\cdot)_{t_0}\over\longrightarrow CH^r(X)\},\]
where the map sends a cycle $Y\in CH^r(T\times X)$ to the difference of refined Gysin
pullbacks $Y_{t_1}-Y_{t_0}\in CH^r(X)$ (\cite[Section 6.2]{Fulton}).
If $k$ is perfect, we may assume $T$ to be smooth and proper (in fact, even a smooth projective curve) by \cite[Proposition 3.14]{ACMVparameter}.
The perfectness of $k$ is needed here basically because regular schemes are not necessarily smooth over non-perfect fields. (See \cite[Lemma 3.8]{ACMVparameter} for this point.)

At least when $k$ is algebraically closed and $X$ is smooth and projective (this is the case classically considered), we can relate $A^r(X)$ to abelian varieties by considering regular homomorphisms:

\begin{defn}[{\cite[Section 2.5]{Samuel}}; see also {\cite[Section 4]{Hartshorne}}, {\cite{KleimanICM}} and {\cite[Definition 1.6.1]{Murre}}]\label{defn: classical regular homomorphism}
Suppose $k$ is algebraically closed. Let $X$ be a connected smooth projective scheme and
let $A$ be an abelian variety. A group homomorphism $\phi\colon A^r(X)\longrightarrow A(k)$ is called {\bf regular} if, for any smooth projective connected scheme $T$ pointed at $t_0$ and for any cycle $Y\in CH^r(T\times X),$ the composition 
\[T(k)\buildrel w_Y\over\longrightarrow A^r(X)\buildrel\phi\over\longrightarrow A(k),\]
where $w_Y$ maps $t\in T(k)$ to $Y_t-Y_{t_0},$ is induced by a scheme morphism $T\longrightarrow A.$
\end{defn}

A regular homomorphism $\phi\colon A^r(X)\longrightarrow A(k)$ is said to be {\bf universal} (\cite[Section 2.5, Remarque (2)]{Samuel}) if for any regular homomorphism $\phi'\colon A^r(X)\longrightarrow A'(k),$ there is a unique homomorphism of abelian varieties $h\colon A\longrightarrow A'$ such that $h(k)\circ\phi=\phi'.$ Universal regular homomorphisms are known to exist in codimension $r=1,2$ and $\dim X$ (\cite[Section 2.5, Remarque (2)]{Samuel}, \cite[Section 1.8 and Theorem A]{Murre}). 

For $r=1,$ it is given by the isomorphism 
\[ w_{\mathcal P}^{-1}\colon A^1(X)\buildrel\cong\over\longrightarrow Pic_{X}^0(k),\]
where $Pic_X^0$ is the Picard variety of $X$ and $\mathcal P\in CH^1(Pic_{X}^0\times X)$ is the divisor corresponding to the Poincar\'e bundle on $Pic_{X}^0\times X.$

The case $r=\dim X$ coincides with the Albanese map
\[alb_X\colon A^{\dim X}(X)\longrightarrow Alb_X(k),\]
i.e. the map that sends $\sum_i n_i\cdot x_i\in A^{\dim X}(X)$ to $\sum_i n_i a_p(x_i)\in  Alb_X(k),$ where $a_p\colon X\longrightarrow Alb_X$ is the canonical map that sends $p\in X(k)$ to $0\in Alb_X(k).$ As $a_p=a_q+a_p(q)$ for any $p$ and $q\in X(k)$ by the universality of Albanese varieties, the Albanese map $alb_X$ is independent of the choice of $p.$

The existence in codimension $r=2$ was proved by Murre (\cite[Theorem A]{Murre}). The proof is based on the works of Merkurjev-Suslin \cite{Merkurjev-Suslin}, Bloch \cite{Bloch, Bloch 2}, Bloch-Ogus \cite{Bloch-Ogus}, Saito \cite{Saito} and Serre \cite{Serre}. If $k=\C,$ a universal regular homomorphism in codimension $2$ can be described in terms of Griffiths's Abel-Jacobi map (see \cite{LiebermanMotive}), which is an analytically defined map
\[AJ\colon CH^2_{hom}(X)\longrightarrow J^2(X)\]
from $CH_{hom}^2(X),$ the kernel of the Betti cycle map, to Griffiths's intermediate Jacobian $ J^2(X).$ Murre (\cite[Theorem C]{Murre}) showed that the restriction of $AJ$ to the algebraic part
\[AJ|_{A^2(X)}\colon A^2(X)\longrightarrow AJ(A^2(X))\]
is universal regular. (In particular, the image of $A^2(X)$ under $AJ$ is the group of rational points of an abelian variety.)

Universal regular homomorphisms have the following properties.

\begin{thm}\label{thm: classical Rojtman-type}
Let $k$ and $X$ be as in Definition~\ref{defn: classical regular homomorphism}. For $r= 1,2$ or $\dim X,$ let $\phi_X^r$ denote the universal regular homomorphism for $A^r(X).$ Then, $\phi_X^1$ is an isomorphism and $\phi_X^{\dim X}$ is an isomorphism on torsion. If $k=\C,$ $\phi_X^2$ is also an isomorphism on torsion.
\end{thm}

\begin{proof}
The case $r=1$ is the theory of Picard varieties (see \cite[Proposition 9.5.10]{Kleiman}). 
The case $r=\dim X$ is known as Rojtman's theorem (\cite{Rojtman, Bloch, Milne}). The final codimension $2$ case is due to Murre (\cite[Theorem C]{Murre}). While the surjectivity is immediate from the construction of $\phi_X^2,$ the injectivity is proved by relating the torsion part of the target of $\phi_X^2$ (which sits in Griffiths's intermediate Jacobian as explained above) to a certain \'etale cohomology group by Artin's comparison theorem 
(\cite[Expos\'e XI, Th\'eor\`eme 4.4]{SGA4}) and then using Merkurjev-Suslin's norm residue isomorphism theorem in degree two (\cite{Merkurjev-Suslin}).
\end{proof}


\subsection{Algebraic part}\label{section: The algebraic part}

Consider a preadditive category $\mathcal A$ equipped with a functor
\[\M\colon \Sm/k\longrightarrow \mathcal A\]
from the category of smooth schemes.
By a preadditive category, we mean a category in which all hom sets are abelian groups and the composition of morphisms is bilinear. What we have in mind is the case where $\mathcal A$ is Voevodsky's triangulated category of motives and $\M$ is the canonical functor which associates with schemes their motives. For affine schemes, we write $\M(R)$ for $\M(\Spec R).$ We put $\Gamma(A):=\Hom_{\mathcal A}(\M(k), A)$ for $A\in \mathcal A.$

For $T\in \Sm/k$ with a rational point $t\in T(k)=\Hom_{\Sm/k}(\Spec k, T),$ we define a map $t^*\colon \Hom_\mathcal A(\M(T), A)\longrightarrow\Gamma(A)$ by sending $Y\in \Hom_\mathcal A(\M(T), A)$ to $Y\circ \M(t)\in \Gamma (A).$

\begin{defn}\label{defn: algebraic part}
Let $\mathfrak T$ be a class of connected quasi-projective schemes in $\Sm/k.$ The subset $\Gamma_\mathfrak T(A)$ of the group $\Gamma(A)$ is defined as
\[\Gamma_\mathfrak T(A):=\bigcup_{\substack{T\in\mathfrak T \\ t_0,t_1\in T(k)}}\mathrm{im}\{\Hom_\mathcal A(\M(T),A)\buildrel t_1^*-t_0^*\over\longrightarrow \Gamma(A)\}.\] 

If $\mathfrak T= alg := \{T\in\Sm/k | \text{ $T$ is connected and quasi-projective.}\},$ $\Gamma_{alg}(A)$ is called the {\bf algebraic part} of $\Gamma(A).$
\end{defn}

\begin{prop}\label{prop: algebraic part}
If $\mathfrak T$ is closed under product, then $\Gamma_{\mathfrak T}(A)$ is a subgroup of $\Gamma(A).$
\end{prop}

\begin{proof}
Since $\Gamma_\mathfrak T(A)$ is clearly closed under taking inverses, it suffices to show that it is closed
under addition. Let $x$ and $x'$ be elements in $\Gamma_\mathfrak T(A).$ Then, for 
some $T\in \mathfrak T,$ $t_0,t_1\in T(k)$ and $Y\in \Hom_{\mathcal A}(\M(T),A)$ we have $x=(t_1^*-t_0^*)(Y).$
Similarly, $x'=(t_1'^*-t_0'^*)(Y')$ holds for some $T'\in\mathfrak T,$ $t_0',t_1'\in T'(k)$ and $Y'\in \Hom_{\mathcal A}(\M(T'),A).$
Now, $x+x'$ is the image of $Y\circ \M(p)+Y'\circ \M(p')$ under
\[(t_1\times t'_1)^*-(t_0\times t'_0)^*\colon Hom_\mathcal A(\M(T\times T'),A)\longrightarrow \Gamma(A),\]
where $p$ (resp., $p'$) is the projection of $T\times T'$ to $T$ (resp., $T'$). Thus, $x+x'$ belongs to $\Gamma_\mathfrak T(A).$
\end{proof}

Notice that the assignment $\mathcal A\ni A\mapsto \Gamma_{\mathfrak T}(A)\in Ab$ is functorial.

Different choices of $\mathfrak T$ may give rise to the same $\Gamma_\mathfrak T(A).$ Crucial for the sequel of this paper is the observation (Proposition~\ref{prop: parametrization by semiabelian varieties}) that the class of semi-abelian varieties is enough to define the algebraic part under Axiom~\ref{axiom} below, which concerns \emph{affine} curves. To state the axiom, we need the notion of multiplicative morphisms. We thank the referee for pointing out a gap in the author's previous proof in Proposition~\ref{prop: an algebraic representative is surjective} and explaining how to fix it by introducing multiplicative morphisms. Axiom~\ref{axiom}, the definition and the use of multiplicative morphisms are due to the referee. The said gap already existed in \cite{Murre}, from which we borrowed many arguments. This point is elaborated in \cite{Kahn-patch}.

\begin{defn}\label{defn: multiplicative}
Let $G$ be a group scheme. Let $A$ be an object in $\mathcal A.$ A morphism $Y\in\Hom_{\mathcal A}(\M(G), A)$ is called {\bf multiplicative} if 
\[Y\circ (\M(p_1)+\M(p_2)-\M(m)-\M(0))=0,\]
where $p_1,p_2,m,0\colon G\times G\longrightarrow G$ are respectively the first projection, the second projection, the multiplication and the zero morphism.
\end{defn}

\begin{axiom}\label{axiom}
Let $C$ be an arbitrary smooth connected affine curve pointed at $c\in C(k).$ Let $Alb_C^0$ be the Serre Albanese variety of $C$ and let $\iota_{c}\colon C\longrightarrow Alb_C^0$ be the canonical morphism that sends $c$ to $0.$ Then, the morphism $\M(\iota_{c})\colon \M(C)\longrightarrow \M(Alb_C^0)$ is a split monomorphism and has a retraction $\rho\colon \M(Alb_C^0)\longrightarrow \M(C)$ that is multiplicative.
\end{axiom}

For the existence of Serre's Albanese varieties over non-perfect fields, we refer the reader to \cite[Theorem A.1]{Wittenberg}.

\begin{rem}
In Lemma~\ref{lem: axiom}, we show that the canonical functor $\Sm/k\longrightarrow \DM_{Nis}^{eff}(k)$ satisfies Axiom~\ref{axiom}.
\end{rem}

Let us quote the following geometric lemma from \cite{ACMVparameter}, which is based on the earlier works of Mumford \cite[Chapter II, Section 6, Lemma]{Mumford}, Jouanolou \cite{Jouanolou}, Poonen \cite{Poonen}, and Charles and Poonen \cite{Charles-Poonen}. To be clear, we repeat our blanket assumptions on schemes in this quote. 

\begin{lem}[{\cite[Theorem A.1 (4)]{ACMVparameter}}]\label{lem: Bertini quote}
Let $X$ be an irreducible separated scheme of dimension $\geq 1$ of finite type over a field $k,$ and let $\bar x_1,\cdots, \bar x_n\in X(\bar k)$ be $\bar k$-points of $X,$ where $\bar k$ is an algebraic closure of $k.$ Then, there is a closed subscheme $C$ of $X$ which is an integral curve over $k$ that contains the images of $\bar x_1,\cdots, \bar x_n.$ Moreover, if $X$ is quasi-projective and smooth over $k,$ there exists such a curve $C$ which is smooth over $k.$
\end{lem}

The above lemma will be used in the following form. This is motivated by {\cite[Proposition 3.10]{ACMVparameter}}, \cite[Example 10.3.2]{Fulton} and \cite[Lemma 9]{Weil}.

\begin{lem}\label{lem: curves are enough}
Let cur be the class of smooth curves. For any $A\in\mathcal A,$ the equality $\Gamma_{alg}(A)=\Gamma_{cur}(A)$ holds.
\end{lem}

\begin{proof}
We need to show the non-trivial inclusion $\Gamma_{alg}(A)\subset \Gamma_{cur}(A).$ Let $x\in \Gamma_{alg}(A).$ There exist $T\in alg,$ $t_0,t_1\in T(k)$ 
and $Y\in \Hom_{\mathcal A}(\M(T), A)$ such that $x=(t_1^*-t_0^*)(Y).$ By Lemma~\ref{lem: Bertini quote}, there is an integral smooth $1$-dimensional subscheme $C$ of $T$ through which $t_0$ and $t_1$ factor. Let $i\colon C\to T$ denote the inclusion and let $t_i'\colon \Spec k\longrightarrow C$ be the morphism with $t_i=i\circ t_i'$ ($i=0,1$). We then have $x=(t_1'^*-t_0'^*)(Y\circ \M(i)).$ Thus, $x\in \Gamma_{cur}(A).$
\end{proof}

Let $\mathfrak G$ be a class of smooth connected group schemes. Since smooth group schemes are quasi-projective by \cite{Chow} (see also \cite[Corollary 1.2]{Conrad}), $\mathfrak G$ satisfies the condition for $\mathfrak T$ in Definition~\ref{defn: algebraic part}. For $G\in\mathfrak G$ and $A\in \mathcal A,$ we define $\Hom_\mathcal A(\M(G),A)_{mult}$ as the subset of $\Hom_\mathcal A(\M(G),A)$ consisting of multiplicative morphisms. We put 
\[\Gamma_{\mathfrak G}^{mult}(A):=\bigcup_{\substack{G\in\mathfrak G \\ g_0,g_1\in G(k)}}\mathrm{im}\{\Hom_\mathcal A(\M(G),A)_{mult}\buildrel g_1^*-g_0^*\over\longrightarrow \Gamma(A)\}.\]

The first consequence of Lemma~\ref{lem: curves are enough} is the following. It is motivated by \cite[Lemma 9]{Weil} and \cite[p.60, Theorem 1]{Lang}.

\begin{prop}\label{prop: parametrization by semiabelian varieties}
Assume Axiom~\ref{axiom}. Let $\mathfrak G$ be any class of smooth connected group schemes that contains the class $\mathfrak J$ of Albanese varieties of pointed smooth affine curves.
Then, for any $A \in \mathcal A,$ the equalities $\Gamma_{alg}(A)=\Gamma_{\mathfrak G}(A)=\Gamma_{\mathfrak G}^{mult}(A)$ hold.
\end{prop}

\begin{proof}
It is enough to show the inclusion $\Gamma_{alg}(A)\subset \Gamma_\mathfrak J^{mult}(A)$ because we trivially have $\Gamma_{alg}(A)\supset\Gamma_{\mathfrak G}(A)\supset\Gamma_{\mathfrak G}^{mult}(A)\supset\Gamma_\mathfrak J^{mult}(A).$ By Lemma~\ref{lem: curves are enough}, it suffices to show
$\Gamma_{cur}(A)\subset \Gamma_{\mathfrak J}^{mult}(A).$

First observe that $\Gamma_{cur}(A)=\Gamma_{\mathit{aff.cur}}(A),$ where $\mathit{aff.cur}$ is the class of smooth affine curves. 
To see the inclusion $\Gamma_{cur}(A)\subset\Gamma_{\mathit{aff.cur}}(A),$ let $x\in\Gamma_{cur}(A).$ Then, there is a smooth curve $C,$ $c_0,c_1\in C(k)$ and
$Y\in \Hom_\mathcal A(\M(C),A)$ such that $x=(c_1^*-c_0^*)(Y).$ If $C$ is not affine, choose an open affine subscheme $U\buildrel i\over\hookrightarrow C$ that contains the images of $c_0$ and $c_1.$ We then have $x=(c_1'^*-c_0'^*)(Y\circ \M(i)),$ where $c_i'\in U(k)$ is the morphism such that $i\circ c_i' = c_i$ ($i=0,1$). Thus, $x\in\Gamma_{\mathit{aff.cur}}(A).$

Now, we need to show $\Gamma_{\mathit{aff.cur}}(A)\subset\Gamma_\mathfrak J^{mult}(A).$ Let $x$ be an element in $\Gamma_{\mathit{aff.cur}}(A).$ Then, there is a smooth affine curve $C,$ rational points $c_0, c_1\in C(k)$ and a morphism $Y\in \Hom_{\mathcal A}(\M(C),A)$ such that $x=(c_1^*-c_0^*)(Y).$ Let $\iota\colon C\longrightarrow Alb_C^0$ be a canonical morphism that is determined once we choose a rational point of $C.$ Let $\rho\colon \M(Alb_C^0)\longrightarrow \M(C)$ be a multiplicative retraction of $\M(\iota),$ which exists under Axiom~\ref{axiom}.
Then, we have $x=((\iota\circ c_1)^*-(\iota\circ c_2)^*)(Y\circ\rho).$
Since $\rho$ is multiplicative, $Y\circ\rho$ is also multiplicative. Therefore, $x\in \Gamma_\mathfrak J^{mult}(A).$
\end{proof}


\subsection {Regular homomorphisms}\label{section: regular homomorphism}

The classical definition of regular homomorphisms (Definition~\ref{defn: classical regular homomorphism}) can be readily generalized to our setting. {\em In the rest of Section~\ref{chapter: Algebraic representatives}, we assume that the base field $k$ is algebraically closed.} We continue using the notation in the previous section.

\begin{defn}\label{defn: regular homomorphism}
Let $A\in \mathcal A$ and let $S$ be a semi-abelian variety.
A group homomorphism $\phi\colon \Gamma_{alg}(A)\longrightarrow S(k)$ is called {\bf regular} if for any $T\in alg,$ 
any rational point $t_0\in T(k)$ and for any $Y\in \Hom_{\mathcal A}(\M(T), A),$ the composition
\[T(k)\buildrel w_Y\over\longrightarrow \Gamma_{alg}(A)\buildrel\phi\over\longrightarrow S(k)\]
is induced by a scheme morphism $f\colon T\longrightarrow S.$ Here, the map $w_Y$ sends
\[t\in T(k)=\Hom_{\Sm/k}(\Spec k,T)\]
to the morphism 
\[t^*(Y)-t_0^*(Y)\colon  \M(k)\buildrel\\M(t)-\M(t_0)\over\longrightarrow \M(T)\buildrel Y\over\longrightarrow A\]
in $\mathcal A.$

For a fixed $A\in \mathcal A,$ a regular homomorphism $\phi\colon\Gamma_{alg}(A)\longrightarrow S(k)$ is said to be {\bf universal} if for any regular homomorphism $\phi'\colon \Gamma_{alg}(A)\longrightarrow S'(k),$ there is a unique homomorphism of semi-abelian $k$-varieties $h\colon S\longrightarrow S'$ such that $h(k)\circ \phi=\phi'.$
\end{defn}

\begin{rem}\label{rem: uniqueness}
The morphism $f\colon T\longrightarrow S$ in the above definition is unique because $T$ is reduced and the set of $k$-points $T(k)$ is Zariski dense in $T.$
\end{rem}

\begin{rem}[Functoriality]\label{rem: functoriality}
Let $\Phi_i\colon \Gamma_{alg}(A_i)\longrightarrow S_i(k)$ ($i=1,2$) be universal regular homomorphisms. Suppose $f\colon A_1\longrightarrow A_2$ is a morphism in $\mathcal A.$ Then, it induces a homomorphism $f_*\colon\Gamma_{alg}(A_1)\longrightarrow \Gamma_{alg}(A_2).$ By the universality, we obtain a unique homomorphism $h\colon S_1\longrightarrow S_2$ of semi-abelian varieties with $h(k)\circ \Phi_1 = \Phi_2\circ f_*.$
\end{rem}

\begin{prop}\label{rem: surjection from a semiabelian variety}
Let $A\in \mathcal A.$ Given a regular homomorphism 
\[\phi\colon\Gamma_{alg}(A)\longrightarrow S(k),\]
there is a semi-abelian variety $S_0$ and a multiplicative morphism $Y_0\in \Hom_\mathcal A(\M(S_0),A)_{mult}$ such that $\mathrm{im}(\phi\circ w_{Y_0})=\mathrm{im}(\phi).$ (
abelian variety is considered to be pointed at the unit.) In particular, there is a semi-abelian subvariety $S'$ of $S$ such that the image of $\phi$ agrees with the group of rational points $S'(k).$
\end{prop}

\begin{proof}
We follow the method of \cite[Proof of Lemma 1.6.2 (i)]{Murre}. Consider the diagram
\[S'(k)\buildrel w_{Y'}\over\longrightarrow \Gamma_{alg}(A) \buildrel\phi\over\longrightarrow S(k)\]
where $S'$ is a semi-abelian variety and $Y'\in \Hom_{\mathcal A}(\M(S'),A).$
Since the composition is induced by a morphism of schemes that sends $0$ of $S'$ to $0$ of $S,$ it is induced by a
homomorphism of semi-abelian varieties by \cite[Theorem 3]{Rosenlicht} (alternatively, see \cite[Theorem 2]{Iitaka} or more recent \cite[Lemma 4.1]{Kahn}; we thank the referee for these three references). Hence, there is a semi-abelian subvariety of $S$ whose group of rational points agrees with $\mathrm{im}(\phi\circ w_{Y'}).$

Choose $S_0$ and multiplicative $Y_0$ such that the dimension of $\mathrm{im}(\phi\circ w_{Y_0})$ is maximal among such diagrams. (By the dimension of $\mathrm{im}(\phi\circ w_{Y_0}),$ we mean the dimension of the underlying semi-abelian variety.)
We claim that they satisfy the desired property $\mathrm{im}(\phi\circ w_{Y_0})=\mathrm{im}(\phi).$

If this is not the case, there is an element $x\in \Gamma_{alg}(A)$ such that $\phi(x)\notin\mathrm{im}(\phi\circ w_{Y_0}).$ 
By Proposition~\ref{prop: parametrization by semiabelian varieties}, there is a semi-abelian variety $S_1$ and a multiplicative morphism 
$Y_1\in \Hom_{\mathcal A}(\M(S_1), A)_{mult}$ with $\phi(x)\in\mathrm{im}(\phi\circ w_{Y_1}).$

Now, let us put 
\[S_2:=S_0\times S_1\] and 
\[Y_2:=Y_0\circ \M(p_0)+Y_1\circ \M(p_1),\]
where $p_i\colon S_2\longrightarrow S_i$ ($i=0, 1$) is the projection. It is straightforward to check that $Y_2$ is multiplicative. 
Let $e\colon S_0(k)\longrightarrow S_2(k)=(S_0\times S_1)(k)$ be the map that sends 
$x\in S_0(k)$ to $(x,0)\in S_2(k).$ Then, we are given the commutative diagram
\begin{displaymath}
\xymatrix{S_0(k) \ar@/^1pc/[drr]^-{w_{Y_0}} \ar[dr]_-e\\
& S_2(k) \ar[r]_-{w_{Y_2}} & \Gamma_{alg}(A) \ar[r]_-\phi & S(k)}
\end{displaymath}   
Therefore, $\mathrm{im}(\phi\circ w_{Y_2})\supset\mathrm{im}(\phi\circ w_{Y_0}).$ 
Similarly, we have $\mathrm{im}(\phi\circ w_{Y_2})\supset\mathrm{im}(\phi\circ w_{Y_1}).$ 
Since $\phi(x)\in \mathrm{im}(\phi\circ w_{Y_1})$ but $\phi(x)\not\in \mathrm{im}(\phi\circ w_{Y_0})$ 
and $\mathrm{im}(\phi\circ w_{Y_2})$ and $\mathrm{im}(\phi\circ w_{Y_0})$ are the groups of rational points of semi-abelian subvarieties (in particular, closed and irreducible subvarieties) of $S,$ these inclusions imply that the 
dimension of $\mathrm{im}(\phi\circ w_{Y_2})$ is strictly greater than that of 
$\mathrm{im}(\phi\circ w_{Y_0}).$ This is a contradiction.
\end{proof}

\begin{prop}\label{prop: an algebraic representative is surjective}
Any universal regular homomorphism is surjective, and any surjective regular homomorphism is surjective on torsion, i.e. the induced homomorphism on the torsion part is surjective.
\end{prop}

\begin{proof}
The first assertion is immediate from Proposition~\ref{rem: surjection from a semiabelian variety}. For the second claim, let $\phi\colon \Gamma_{alg}(A)\longrightarrow S(k)$ be a surjective regular homomorphism. By Proposition~\ref{rem: surjection from a semiabelian variety}, there is a semi-abelian variety $S_0$ and a multiplicative morphism $Y_0\colon \M(S_0)\longrightarrow A$ such that the composition $\phi\circ w_{Y_0}$
\[S_0(k)\buildrel w_{Y_0}\over\longrightarrow \Gamma_{alg}(A)\buildrel\phi\over\longrightarrow S(k)\]
is surjective. Since the kernel of $\phi\circ w_{Y_0}$ is an extension of a finite group 
by a divisible group, we have $\mathrm{ker}(\phi\circ w_{Y_0})\otimes \Q/\Z=0.$ This implies that $\phi\circ w_{Y_0}$ is surjective on torsion. Now, since $Y_0$ is multiplicative, $w_{Y_0}$ is a group homomorphism. Hence, $\phi$ is also surjective on torsion.
\end{proof}


\subsection{The existence criterion of universal regular homomorphisms}\label{subsection: The existence criterion of universal regular homomorphisms}

We prove a criterion for the existence of universal regular homomorphisms (Proposition~\ref{prop: existence criterion}). This is a generalization of \cite[Theorem 2.2]{Saito} as presented in \cite[Proposition 2.1]{Murre} (cf. \cite[Th\'eor\`eme 2]{Serre}). Our proof is a combination of the methods of Serre \cite{Serre}, Saito \cite{Saito} and Murre \cite{Murre}. We thank the referee for pointing out that the argument used in the proofs of Lemma~\ref{lem: maximal homomorphism factorization} and Proposition~\ref{prop: existence criterion} also appeared in \cite[Chapter II, Section 3, p. 43--44]{Lang}. The following definition is due to Serre, but also implicit in {\em loc. cit.}

\begin{defn}[{\cite[D\'efinition 2]{Serre}}]\label{defn: maximal homomorphism}
Let $A\in \mathcal A.$ A regular homomorphism $\phi: \Gamma_{alg}(A)\longrightarrow S(k)$ is called a {\bf maximal homomorphism} if it is surjective and satisfies the following condition:

Given any factorization 
\begin{displaymath}
\xymatrix{& S'(k) \ar[d]^-{\pi(k)}\\
\Gamma_{alg}(A) \ar[r]_-\phi \ar[ur]^-{\phi'} & S(k),} 
\end{displaymath}
where $\phi'$ is regular and $\pi\colon S'\longrightarrow S$ is an isogeny, $\pi$ is in fact an isomorphism.
\end{defn}

\begin{lem}\label{lem: maximal homomorphism factorization}
Let $\phi\colon \Gamma_{alg}(A)\longrightarrow S(k)$ be a regular homomorphism. Then, there is a factorization 
\begin{displaymath}
\xymatrix{ \Gamma_{alg}(A) \ar[rr]^-\phi \ar[dr]_-g && S(k) \\
& S'(k) \ar[ur]_-{h(k)}}
\end{displaymath}
such that $g$ is a maximal homomorphism and $h\colon S'\longrightarrow S$ is a finite morphism.
\end{lem}

\begin{proof}
We follow the proof of \cite[Th\'eor\`eme 1]{Serre}. By Proposition~\ref{rem: surjection from a semiabelian variety}, we may assume that $\phi$ is surjective. 
If $\phi$ is maximal, there is nothing to prove.

If $\phi$ is not maximal, there is a factorization
\begin{displaymath}
\xymatrix{& S_1(k) \ar[d]^-{\pi_1(k)}\\
\Gamma_{alg}(A) \ar@{->>}[r]_-\phi \ar[ur]^-{\phi_1} & S(k),} 
\end{displaymath}
such that $\phi_1$ is regular and $\pi_1$ is an isogeny that is not an isomorphism. 
If $\phi_1$ is maximal, there is nothing more to do. If it is not, repeat this process. 

Suppose we obtain an infinite tower
\begin{displaymath}
\xymatrix{& \vdots\ar[d]^{\pi_3(k)}\\
& S_2(k) \ar[d]^{\pi_2(k)}\\
& S_1(k) \ar[d]^{\pi_1(k)}\\
\Gamma_{alg}(A) \ar@{->>}[r]_-\phi \ar[ur]_-{\phi_1} \ar[uur]^-{\phi_2}  & S(k),} 
\end{displaymath}
where $\phi_i$'s are regular and $\pi_i$'s are isogenies but not isomorphisms. 
Using Proposition~\ref{rem: surjection from a semiabelian variety}, choose a semi-abelian 
variety $S_0$ and $Y_0\in Hom_{\mathcal A}(\M(S_0),A)$ such that $\phi\circ w_{Y_0}$ is surjective. Then, for any $i,$ since $\pi_i$ is an isogeny, $\phi_i\circ w_{Y_0}$ is surjective.

Therefore, we obtain homomorphisms between function fields
\begin{displaymath}
\xymatrix{& \vdots \\
& K(S_2) \ar@{_{(}->}[ddl] \ar@{^{(}->}[u]_{j_3}\\
& K(S_1) \ar@{_{(}->}[dl]    \ar@{^{(}->}[u]_{j_2}\\
K(S_0) & K(S)  \ar@{_{(}->}[l] \ar@{^{(}->}[u]_{j_1}}
\end{displaymath}
where none of the $j_i$'s are isomorphisms. Now, the extension $\bigcup_{i\geq1}K(S_i)/ K(S)$ cannot be finitely generated, but $K(S_0)$ is a finitely generated field over $K(S).$ This is a contradiction by \cite[Lemme 1]{Serre}.
\end{proof}

\begin{prop}\label{prop: existence criterion}
Consider a fixed object $A\in\mathcal A.$ There is a universal regular homomorphism for $\Gamma_{alg}(A)$ if and only if there is a constant $c$ such that for any maximal homomorphism $\phi\colon \Gamma_{alg}(A)\longrightarrow S(k),$ the inequality $\dim S\leq c$ holds. In fact, a maximal homomorphism with a maximal dimensional target is a universal regular homomorphism.
\end{prop}

\begin{proof}
$``\Rightarrow"$ is clear. We prove the converse by combining the arguments in \cite[Th\'eor\`eme 2]{Serre} and \cite[Proposition 2.1]{Murre}. Let $\phi\colon \Gamma_{alg}(A)\longrightarrow S(k)$ be a maximal homomorphism with a maximal dimensional target $S.$ 
Suppose $\phi'\colon \Gamma_{alg}(A)\longrightarrow S'(k)$ is another regular homomorphism. 

Since $\phi\times\phi'\colon \Gamma_{alg}(A)\longrightarrow (S\times S')(k)$ is clearly a regular homomorphism, Lemma~\ref{lem: maximal homomorphism factorization} gives a factorization
\[\phi\times\phi'\colon \Gamma_{alg}(A)\buildrel g\over\longrightarrow S''(k)\buildrel i(k)\over\longrightarrow (S\times S')(k)\]
with a maximal homomorphism $g$ and a finite morphism $i\colon S''\longrightarrow S\times S'.$ Consider the commutative diagram
\begin{displaymath}
\xymatrix{& & S(k) \\
\Gamma_{alg}(A) \ar@{->>}[r]^-{g} \ar@/^/@{->>}[urr]^-{\phi} \ar@/_/[drr]_-{\phi'} & S''(k) \ar[r]^-{i(k)}   & (S\times S')(k) \ar[u]_-{p(k)} \ar[d]^-{p'(k)}\\
& & S'(k)}
\end{displaymath}
where $p$ (resp., $p'$) is the projections $S\times S'\longrightarrow S$ (resp., $S\times S'\longrightarrow S'$).

Since $\phi$ is surjective, $p\circ i$ is surjective. Since $S$ has the maximal dimension, we must have $\dim S=\dim S''.$ 
Therefore, $p\circ i$ is an isogeny. Now, since $\phi$ is a maximal homomorphism, $p\circ i$ is an isomorphism. Let us put
\[r:=(p\circ i)^{-1}:S\longrightarrow S'',\]
and define $h:=p'\circ i\circ r\colon S\longrightarrow S'.$ We then have
\[h(k)\circ\phi = p'(k)\circ i(k)\circ r(k)\circ \phi = p'(k)\circ i(k)\circ g= \phi'.\]

Lastly, $h$ is the only scheme morphism for which  $\phi'=h(k)\circ\phi$ holds because $\phi$ is surjective and a morphism from $S$ to $S'$ is uniquely determined on $S(k),$ which is Zariski dense in $S$ (see Remark~\ref{rem: uniqueness}).
\end{proof}

\begin{cor}\label{cor: existence criterion with surjective regular homomorphisms}
Consider a fixed object $A\in\mathcal A.$ There is a universal regular homomorphism for $\Gamma_{alg}(A)$ if and only if there is a constant $c$ such that for any surjective regular (not necessarily maximal) homomorphism $\phi\colon \Gamma_{alg}(A)\longrightarrow S(k),$ the inequality $\dim S\leq c$ holds.
\end{cor}

\begin{proof}
It is immediate from Proposition~\ref{prop: existence criterion}.
\end{proof}

We finish this section with a more practical criterion. The statement and proof of the following corollary are the result of considerable improvement by the referee of the author's earlier version, which unnecessarily involved infinitely many prime numbers $l.$

\begin{cor}\label{cor: practical criterion}
Consider a fixed object $A\in\mathcal A.$ Suppose there exists a prime number $l$ different from $\mathrm{char}~k$ for which the $l$-torsion part $_l \Gamma_{alg}(A)$ of $\Gamma_{alg}(A)$ is finite. Then, $\Gamma_{alg}(A)$ has a universal regular homomorphism. 
\end{cor}

\begin{proof}
Let $\phi\colon\Gamma_{alg}(A)\longrightarrow S(k)$ be an arbitrary surjective regular homomorphism. Since it is surjective on torsion (Proposition~\ref{prop: an algebraic representative is surjective}), it induces a surjective homomorphism on the $l$-primary torsion part 
\[\phi\{l\}\colon\Gamma_{alg}(A)\{l\}\longrightarrow S(k)\{l\}.\]
Now, if $_l \Gamma_{alg}(A)$ is finite, $\Gamma_{alg}(A)\{l\}$ is a direct sum of a finite group and an $l$-divisible torsion group of finite corank. (To see this, write $\Gamma_{alg}(A)\{l\}$ as a direct sum of a divisible group and a reduced group and apply \cite[Theorem 9]{Kaplansky}.) Therefore, the corank of $S(k)\{l\}$ is bounded by the corank of $\Gamma_{alg}(A)\{l\}.$ Since $l\neq \mathrm{char}~k,$ we have $\dim S\leq \mathrm{corank}~S(k)\{l\}.$ Now the corollary follows from Corollary~\ref{cor: existence criterion with surjective regular homomorphisms}.
\end{proof}


\section{The case of motives}\label{section: The case of motives}

We apply the results of the previous section to the triangulated category of Voevodsky motives. We take $\mathcal A$ to be the closed symmetric monoidal triangulated category $\DM_{Nis}^{eff}(k)$ defined in \cite[Definition 11.1.1]{Cisinski-Deglise} and $\mathcal M \colon \Sm/k\longrightarrow \mathcal A$ to be the monoidal functor that sends $X\in\Sm/k$ to the motive $M(X).$ We follow the notation introduced in Section~\ref{section: introduction}.


\subsection{Algebraic part}\label{section: algebraic part motives}

For $X\in \Sch/k,$ we put $H_0(X,\Z):=\Hom_{\DM_{Nis}^{eff}(k)}(\Z,M(X)).$ This is isomorphic to Suslin's algebraic singular homology in degree $0$ by \cite[Proposition 14.18]{MVW}. If $X$ is connected, the structure morphism of $X$ induces the degree map
\[deg\colon H_0(X,\Z) \longrightarrow \Hom_{\DM_{Nis}^{eff}(k)}(\Z,\Z)\cong\Z.\]
The kernel of $deg$ is denoted by $H_0(X,\Z)^0.$

\begin{prop}[{\cite[Lemma 7.10]{Bloch-Ogus},\cite[Lemma 9]{Weil}}]\label{prop: algebraic part is divisible}
If $k$ is algebraically closed, then $\Gamma_{alg}(A)$ is divisible for any $A\in \DM_{Nis}^{eff}(k).$
\end{prop}

\begin{proof}
By Lemma~\ref{lem: curves are enough}, we have the equality
\[\Gamma_{alg}(A)=\bigcup_{C, ~{\text{smooth curve}}}\mathrm{im}\{H_0(C,\Z)^0\otimes \Hom_{\DM_{Nis}^{eff}(k)}(M(C),A)\buildrel P\over\longrightarrow \Gamma(A)\},\]
where $P$ is defined by the composition of morphisms in $\DM_{Nis}^{eff}(k).$ Thus, any element of $\Gamma_{alg}(A)$ 
is an image of some element of $H_0(C,\Z)^0$ for some smooth curve $C.$ Now, $H_0(C,\Z)^0$ is divisible because
it is isomorphic to $Alb_C^0(k)$ by \cite[Theorem 3.4.2~(3)]{VSF5}. Therefore, $\Gamma_{alg}(A)$ is also divisible.
\end{proof}

We recover Fulton's algebraic equivalence at least for smooth schemes.

\begin{prop}\label{prop: comparison with the classical algebraic part}
Suppose $k$ is perfect and let $X\in\Sm/k.$ Let $A^r(X)$ be the algebraic part of $CH^r(X)$ in the sense of Fulton (see Subsection~\ref{section: The case of Chow groups}). Then, there is a natural isomorphism
\[\Gamma_{alg}(\underline{\Hom}(M(X),\Z(r)[2r]))\buildrel\cong\over\longrightarrow A^r(X).\]
\end{prop}

\begin{proof}
This isomorphism is induced by the natural isomorphism
\[F\colon  \Hom_{\DM_{Nis}^{eff}(k)}(M(X),\Z(r)[2r])\longrightarrow CH^r(X),\]
in \cite[Theorem 19.1]{MVW}. The contravariant naturality of $F$ implies the commutativity of the diagram
\begin{displaymath}\xymatrixcolsep{-0.1cm}\xymatrixrowsep{1cm}
\xymatrix{\Hom_{\DM_{Nis}^{eff}(k)}(M(T), \underline{\Hom}(M(X),\Z(r)[2r]) \ar[d]_{\text{adjunction}}^\cong \ar[rd]^-{t_1^*-t_0^*}  \\
\Hom_{\DM_{Nis}^{eff}(k)}(M(T)\otimes M(X),\Z(r)[2r]) \ar[d]_F^\cong \ar[rd]_-{t_1^*\otimes id_{M(X)}-t_0^*\otimes id_{M(X)}} & \Hom_{\DM_{Nis}^{eff}(k)}(\Z,\underline{\Hom}(M(X),\Z(r)[2r])) \ar[d]^{\text{adjunction}}_\cong  \\
CH^r(T\times X) \ar[rd]_-{(\cdot)_{t_1}-(\cdot)_{t_0}} & \Hom_{\DM_{Nis}^{eff}(k)}(\Z\otimes M(X),\Z(r)[2r]) \ar[d]^F_\cong \\
& CH^r(X)}
\end{displaymath}
for any smooth scheme $T$ and any $t_0$ and $t_1\in T(k).$ Therefore, $F$ induces an isomorphism $\Gamma_{alg}(\underline{\Hom}(M(X),\Z(r)[2r]))\longrightarrow A^r(X).$
\end{proof}

\begin{rem}\label{prop: comparison with the classical algebraic part under resolution}
(This remark was originally stated for an equidimensional scheme $X.$ We thank the referee for pointing out that the equidimentionality is unnecessary if we work with a Chow group indexed by dimension, not by codimension.)

Let us assume the resolution of singularities. In this case, for an arbitrary $X\in \Sch/k,$ we have a comparison isomorphism between the
Borel-Moore homology $H_{2r}^{BM}(X,\Z(r))\buildrel\text{def}\over = \Hom_{\DM_{Nis}^{eff}(k)}(\Z(r)[2r], M^c(X))$ and the Chow group $CH_{r}(X)$ (\cite[Proposition 19.18]{MVW} is stated for equidimensional $X$; see \cite[Theorem 5.3.14]{Kelly} for a statement for general $X$).
Thus, we may expect that it induces an isomorphism 
\[\Gamma_{alg}(\underline{\Hom}(\Z(r)[2r],M^c(X)))\buildrel\cong\over\longrightarrow A_{r}(X),\]
where $A_r(X)$ is now defined as
\[A_r(X)=\bigcup_{\substack{T\in alg\\ t_0,t_1\in T(k)}}\mathrm{im}\{CH_{r+\dim T}(T\times X)\buildrel (\cdot)_{t_1}-(\cdot)_{t_0}\over\longrightarrow CH_r(X)\},\]
where $(\cdot)_{t_{i}}$ is a refined Gysin homomorphism (\cite[Section 6.2]{Fulton}).

For this to be true, the following diagram needs to commute for all $T\in alg$ and $t_0, t_1\in T(k)$:
\begin{displaymath}\xymatrixcolsep{-0.6cm}
\xymatrix{\Hom_{\DM_{Nis}^{eff}(k)}(M(T), \underline{\Hom}(\Z(r)[2r],M^c(X)))   \ar[rd]^-{t_1^*-t_0^*} \\
 CH_{r+\dim T}(T\times X) \ar[rd]_-{(\cdot)_{t_1}-(\cdot)_{t_0}}  \ar@{}[u]|-{\rotatebox{90}{$\cong$}} & \Hom_{\DM_{Nis}^{eff}(k)}(\Z, \underline{\Hom}(\Z(r)[2r],M^c(X))) \\
& CH_r(X) \ar@{}[u]|-{\rotatebox{90}{$\cong$}} ,}
\end{displaymath}
where the right vertical isomorphism is given by the comparison isomorphism mentioned in the previous paragraph, and the left is given by a comparison isomorphism together with a duality isomorphism in \cite[Theorem 8.2]{Friedlander-Voevodsky}. However, the author does not know if this diagram actually commutes.
\end{rem}

The following proposition on zero cycles is a generalization of \cite[Proposition 3.11]{ACMVparameter}, which proved the statement for the Chow groups of smooth integral projective schemes. 

\begin{prop}\label{prop: algebraic part of zero cycles}
Suppose $k$ is perfect and let $X\in \Sch/k.$ Then, there is an inclusion $\Gamma_{alg}(M(X))\subset H_0(X,\Z)^0.$
If $X$ is smooth, quasi-projective and connected (i.e. if $X\in alg$), then this inclusion is an equality. If $k$ is algebraically closed, the equality holds for any irreducible $X.$
\end{prop}

\begin{proof}
Let $str_S$ denote the structure morphism of $S\in \Sch/k.$ To show the inclusion $\Gamma_{alg}(M(X))\subset H_0(X,\Z)^0,$ we need to prove that for any $T\in alg,$ any $t_0$ and $t_1\in T(k)$ and any $Y\in \Hom_{\DM_{Nis}^{eff}(k)}(M(T),M(X)),$ the composition
$Y\circ (M(t_1)-M(t_0))$ belongs to $H_0(X,\Z)^0;$ in other words, the composition
\[ \Z\buildrel{M(t_1)-M(t_0)}\over\longrightarrow M(T)\buildrel Y\over\longrightarrow M(X)\buildrel M(str_X)\over\longrightarrow\Z\]
in $\DM_{Nis}^{eff}(k)$ is zero.

Indeed, since $T$ is smooth and connected, the group $\Hom_{\DM_{Nis}^{eff}(k)}(M(T),\Z)\cong \Z$ is generated
by $M(str_T).$ Thus, there is an integer $n$ such that $M(str_X)\circ Y=n\cdot M(str_T).$ 
Therefore, $M(str_X)\circ Y\circ (M(t_1)-M(t_0))=n\cdot M(str_T)\circ (M(t_1)-M(t_0))=n\cdot (id_\Z-id_\Z)=0.$

Next, let us prove the inclusion in the other direction. Let us assume that $X$ is a smooth curve for the moment. Let $x$ be an element of $H_0(X,\Z)^0.$ It is represented by the difference $z-w$ of two effective zero cycles $z$ and $w$ of the same degree, say, $d.$ 

Let $X^d$ (resp. $X^{(d)}$) be the $d$-fold product (resp., symmetric product) of $X.$ By \cite[Proposition 3.2]{MilneJV}, $X^{(d)}$ is a smooth quasi-projective $k$-scheme. Let $D_i \subset X^d\times X$ be the graph of the $i$-th projection $p_i\colon X^d\longrightarrow X.$ Now the finite correspondence $D=\sum_{i=1}^dD_i\in Cor_k(X^d,X)$ in the sense of \cite[Definition 1.1]{MVW} descends to the finite correspondence $D_{can}\in Cor_k(X^{(d)},X)$ by \cite[Example 3.12]{MilneJV}. This defines a morphism, which we shall write with the same symbol, $D_{can}\in \Hom_{\DM_{Nis}^{eff}(k)}(M(X^{(d)}),M(X)).$ Now, by [Ibid., Theorem 3.13], there exist rational points $Z$ and $W\in X^{(d)}(k)$ such that $D_{can}\circ M(Z)=z$ and $D_{can}\circ M(W)=w$ in $\DM_{Nis}^{eff}(k)$. Hence, we have $(Z^*-W^*)(D_{can})=z-w=x.$ Therefore, $x$ belongs to $\Gamma_{alg}(M(X)).$
 
Now, let us treat the higher dimensional case. Notice that if $X\in\Sch/k$ is irreducible and $\dim X\geq 2,$ then, for any $x\in H_0(X,\Z)^0,$ there is an integral curve $C$ on $X$ that contains the support of a zero cycle representing $x$ by Lemma~\ref{lem: Bertini quote}. $C$ can be chosen to be smooth if $X$ is additionally smooth and quasi-projective.

Hence, for the case $X\in alg,$ the functoriality of algebraic part, the already proved first claim of the proposition and the case of smooth curves give the commutative diagram
\begin{displaymath}
\xymatrix{\Gamma_{alg}(M(X)) & \subset & H_0(X,\Z)^0\\
\Gamma_{alg}(M(C)) \ar[u] & = & H_0(C,\Z)^0 \ar[u], }
\end{displaymath}
where the vertical arrows are induced by the inclusion $C\hookrightarrow X.$ Now, by the choice of $C,$ $x$ is in the image of $H_0(C,\Z)^0.$ This implies $x\in \Gamma_{alg}(M(X)).$

The remaining case of $X$ only irreducible but $k$ algebraically closed follows by replacing $C$ in the above diagram with its normalization $\tilde C$ because zero cycles on $C$ can be lifted to $\tilde C$ when the base field is algebraically closed.
\end{proof}

\begin{rem}\label{rem: counterexample}
If Remark~\ref{prop: comparison with the classical algebraic part under resolution} is affirmative, \cite[Example 3.13]{ACMVparameter} provides a singular scheme $X$ over a non-algebraically closed field for which the inclusion $\Gamma_{alg}(M(X))\subset H_0(X,\Z)^0$ is strict.
\end{rem}


\subsection{Unpointed regular homomorphisms}\label{section: unpointed regular homomorphism}
 
The following main definition, formulation and ideas are due to the referee. This subsection should be attributed to him or her. As in Subsection~\ref{section: regular homomorphism}, {\em we assume the base field $k$ to be algebraically closed in this subsection.} 

We consider an ``unpointed" version of regular homomorphisms, which is defined on the whole $\Gamma(A).$ The unpointed version provides a natural framework to relate universal regular homomorphisms (on the algebraic part) and $1$-motivic theories such as \cite{Ramachandran} and \cite{Ayoub-Barbieri-Viale}. In order to deal with non-algebraic part, we need a larger 
class of group schemes as in \cite{Ramachandran}.

\begin{defn}\label{defn: semi-abelian scheme}
A group scheme $G$ locally of finite type over $k$ is called a {\bf semi-abelian scheme} if its identity component is a semi-abelian variety. 
\end{defn}

\begin{defn}\label{defn: unpointed regular homomorphism}
Let $\mathcal M\colon \Sm/k\longrightarrow \mathcal A$ be a functor as in Subsection~\ref{section: The algebraic part} and let $G$ be a semi-abelian scheme. A group homomorphism $\phi\colon \Gamma(A)\longrightarrow G(k)$ is called an {\bf unpointed regular homomorphism} if for any $T\in \Sm/k$ (not necessarily connected nor quasi-projective) and 
any $Y\in \Hom_{\mathcal A}(\mathcal M(T), A),$ the composition 
\[T(k)\buildrel Y_*\over\longrightarrow \Gamma(A)\buildrel\phi\over\longrightarrow G(k),\]
is induced by a morphism $T\longrightarrow G$ of schemes. Here, $Y_*$ sends $t\in T(k)$ to $Y\circ \mathcal M(t).$ 
\end{defn}

The morphism $T\longrightarrow G$ in the above definition is unique if it exists (see Remark~\ref{rem: uniqueness}). Although the definition is stated for an abstract functor $\mathcal M,$ we only consider the case $\mathcal M = M\colon \Sm/k\longrightarrow \DM_{Nis}^{eff}(k)$ in this paper. 

Let $\Sh_{Nis}^{tr}(\Sm/k)$ denote the category of Nisnevich sheaves with transfers (\cite[Lecture 13]{MVW}) and let $\HI_{Nis}(k)$ be the full subcategory of homotopy invariant Nisnevich sheaves with transfers ([Ibid., Definition 2.15]). For $X\in\Sm/k,$ we write $\tilde X:=\Hom_{\Sm/k}(-,X)$ for the presheaf of sets on $\Sm/k$ represented by $X.$ If $G$ is a semi-abelian scheme, $\tilde G$ belongs to $\HI_{Nis}(k)$ by \cite[Lemma 1.4.4]{Barbieri-Viale-Kahn} (see also \cite[Lemma 3.2]{Spiess-Szamuely}). For any $X\in\Sch/k,$ the cohomology sheaf $h_m^{Nis}(X):= H^m(C_*\Z_{tr}(X))_{Nis}$ ($m\in\Z$) of the Suslin complex $C_*\Z_{tr}(X)$ also belongs to $\HI_{Nis}(k)$ (see \cite[Lemma 3.2.1]{VSF5}).

\begin{lem}\label{lem: link}
Let $A\in \HI_{Nis}(k)$ and let $G$ be a semi-abelian scheme. Then, taking sections over $k$ gives rise to a canonical isomorphism of abelian groups
\[g\colon \Hom_{\HI_{Nis}(k)}(A,\tilde G)\buildrel\cong\over\longrightarrow \{\Gamma(A)\buildrel\phi\over\longrightarrow G(k)|\text{ $\phi$ is unpointed regular}\}.\]
Furthermore, if $A= h_0^{Nis}(X)$ ($X\in\Sch/k$), the canonical morphism $M(X)\longrightarrow A$ in $\DM_{Nis}^{eff}(k)$ induces a bijection
\[\{\Gamma(A)\longrightarrow G(k)|\text{ unpointed regular}\}\]
\[\longrightarrow \{\Gamma(M(X))\longrightarrow G(k)|\text{ unpointed regular}\}.\]
(The groups $\Gamma(h_0^{Nis}(X))$ and $\Gamma(M(X))$ are canonically isomorphic as explained in the proof below. The point of the second claim is that the definition of regularity depends on the motives $h_0^{Nis}(X)$ and $M(X).$) 
\end{lem}

\begin{proof}
Let $T\in \Sm/k$ and $Y\in \Hom_{\DM_{Nis}^{eff}(k)}(M(T),A).$ Since $A$ is a homotopy invariant Nisnevich sheaf with transfers, 
$Y$ can be regarded as a morphism $Y'\colon \tilde T\longrightarrow A$ of Nisnevich sheaves of sets under the isomorphisms
\begin{eqnarray}
\Hom_{\DM_{Nis}^{eff}(k)}(M(T),A)&\cong& \Hom_{D(\Sh_{Nis}^{tr}(\Sm/k))}(\Z_{tr}(T),A)\nonumber\\
&\cong& A(T)\nonumber\\
&\cong& \Hom_{\Sh_{Nis}(\Sm/k;\mathrm{Set})}(\tilde T,A),\nonumber
\end{eqnarray}
where the first isomorphism follows from \cite[Eq. (14.6.1)]{MVW}.

Let us first show that the map $g$ is well-defined; namely, if $\tilde \phi\in \Hom_{\HI_{Nis}(k)}(A,\tilde G),$ then $\tilde\phi(k)\colon \Gamma(A)\longrightarrow G(k)$ is an unpointed regular homomorphism. We need to show, for $T$ and $Y$ as above, that the composition 
$T(k)\buildrel Y_*\over\longrightarrow \Gamma(A)\buildrel \tilde\phi(k)\over\longrightarrow G(k)$
is induced by a scheme morphism $T\longrightarrow G.$ But, this is the section over $k$ of the morphism $\tilde T\buildrel Y'\over\longrightarrow A\buildrel\tilde\phi\over\longrightarrow \tilde G$ of sheaves of sets. By Yoneda's lemma, this is induced by a scheme morphism $T\longrightarrow G.$

Next, let us construct the inverse of $g.$ Suppose $\phi\colon \Gamma(A)\longrightarrow G(k)$ is an unpointed regular homomorphism. For a smooth scheme $T,$ define 
\[f_T\colon \Hom_{\DM_{Nis}^{eff}(k)}(M(T),A)\longrightarrow \tilde G(T)\]
by sending $Y\in \Hom_{\DM_{Nis}^{eff}(k)}(M(T),A)$ to the unique scheme morphism $T\longrightarrow G$ that induces the composition $T(k)\buildrel Y_*\over\longrightarrow \Gamma(A)\buildrel\phi\over\longrightarrow G(k).$ Here, the uniqueness follows because $T(k)$ is Zariski dense in $T$ (see Remark~\ref{rem: uniqueness}). Observe that $f_T$ is natural in $T$ with respect to finite correspondences, i.e. the following diagram commutes for any smooth scheme $T'$ and any finite correspondence $\gamma\in\Cor_k(T',T):$
\begin{displaymath}
\xymatrix{ \Hom_{\DM_{Nis}^{eff}(k)}(M(T),A) \ar[r]^-{f_T} \ar[d]_-{\gamma^*}& \tilde G(T) \ar[d]_-{\gamma^*} \\
\Hom_{\DM_{Nis}^{eff}(k)}(M(T'),A) \ar[r]^-{f_{T'}} & \tilde G(T').}
\end{displaymath}

Now, composing $f_T$ with the canonical isomorphism $A(T)\cong  \Hom_{\DM_{Nis}^{eff}(k)}(M(T),A),$ we obtain a morphism $\psi\colon A\longrightarrow \tilde G$ in $HI_{Nis}(k).$ It is routine to check that the assignment $\phi\mapsto\psi$ is inverse to $g.$ This proves the first half of the lemma.

For the second claim, consider the diagram
\begin{displaymath}
\xymatrix{\Hom_{\HI_{Nis}(k)}(h_0^{Nis}(X),\tilde G)  \ar[r]^-g \ar[d]_-{q^*}& \{\Gamma(h_0^{Nis}(X))\longrightarrow G(k)|\text{ unpointed regular}\} \ar[d]\\
\Hom_{\Sh_{Nis}^{tr}(k)}(\Z_{tr}(X),\tilde G) \ar[r]^-h& \{\Gamma(M(X))\longrightarrow G(k)|\text{ group homomorphism}\},}
\end{displaymath}
where both vertical arrows are induced by the canonical morphism $q\colon \Z_{tr}(X)\longrightarrow h_0^{Nis}(X)$ and $h$ is the composition of the canonical maps
\[\Hom_{\Sh_{Nis}^{tr}(k)}(\Z_{tr}(X),\tilde G)\longrightarrow \Hom_{\DM_{Nis}^{eff}(k)}(M(X),\tilde G)\]
and 
\[\Hom_{\DM_{Nis}^{eff}(k)}(M(X),\tilde G)\longrightarrow \{\Gamma(M(X))\longrightarrow G(k)|\text{ group homomorphism}\}.\]

First, notice that the right vertical arrow is injective because it is the pre-composition with the following series of canonical isomorphisms 
\[\Gamma(M(X))=\Hom_{\DM_{Nis}^{eff}(k)}(\Z,M(X))\cong\Hom_{D(\Sh_{Nis}^{tr}(k))}(\Z,C_*\Z_{tr}(X))\cong h_0^{Nis}(X)(k).\]
Here, the first isomorphism follows because $C_*\Z_{tr}(X)$ is an $\mathbb A^1$-local object (\cite[5.2.18]{Cisinski-Deglise}, \cite[Theorem 13.8]{MVW}) and the second is induced by $q.$

Second, notice that the image of the map $h$ consists of unpointed regular homomorphisms. This follows from the well-definedness of $g$ (i.e. the second paragraph of this proof) and the fact that any sheaf morphism $\Z_{tr}(X)\longrightarrow\tilde G$ factors through $q$ (\cite[Example 2.20]{MVW}).

Therefore it remains to show that $h\circ q^*$ is surjective onto the subgroup of unpointed regular homomorphisms from $\Gamma(M(X))$ to $G(k).$ Let $\phi'\colon \Gamma(M(X))\longrightarrow G(k)$ be an unpointed regular homomorphism. The construction of $f_T$ in the first part of the proof works verbatim for $\phi'$ and yields a morphism
\[f_{(-)}\colon \Hom_{\DM_{Nis}^{eff}(k)}(M(-),M(X)) \longrightarrow \tilde G\]
of homotopy invariant \emph{presheaves} with transfers. (For $T\in\Sm/k,$ $f_T(Y)$ is the morphism $T\longrightarrow G$ that induces $\phi'\circ Y_*.$) Taking the Nisnevich sheafification, we obtain a morphism in $\HI_{Nis}(k)$
\[h_0^{Nis}(X)\buildrel\text{def}\over=H^0( C_*\Z_{tr}(X))_{Nis}\longrightarrow \tilde G,\]
but this is by construction the preimage of $\phi'$ under $h\circ q^*.$
\end{proof}

Let $G^0$ denote the identity component of a semi-abelian scheme $G.$ Any unpointed regular homomorphism $\phi\colon \Gamma(A)\longrightarrow G(k)$ can be restricted to a regular homomorphism $\phi_{alg}\colon \Gamma_{alg}(A)\longrightarrow G^0(k).$ Indeed, for any connected $T\in\Sm/k$ pointed at any $t_0\in T(k)$ and any morphism $Y\in \Hom_{\DM_{Nis}^{eff}(k)}(M(T),A),$ the composition 
\[T(k)\buildrel w_Y\over\longrightarrow \Gamma_{alg}(A)\buildrel\phi\over\longrightarrow G(k)\]
is induced by a scheme morphism because $\phi\circ w_Y$ is the difference of the morphism that sends $t\in T(k)$ to $\phi(Y\circ M(t))$ and the constant morphism with value $\phi(Y\circ M(t_0)).$ Since $\phi\circ w_Y$ sends $t_0$ to zero in $G(k)$ and the image of $T$ under any scheme morphism is connected, the image of $T(k)$ lies in $G^0(k).$ As $\Gamma_{alg}(A)$ is generated by the image of $T(k)$ under $w_Y$ (where $T$ and $Y$ vary), the image of $\Gamma_{alg}(A)$ is contained in $G^0(k).$

Conversely, any regular homomorphism gives rise to an unpointed regular homomorphism. We thank the referee for the following construction using the divisibility of algebraic part. 

\begin{lem}\label{lem: construct unpointed regular from regular}
Let $A\in \DM_{Nis}^{eff}(k)$ and let $\psi\colon \Gamma_{alg}(A)\longrightarrow S(k)$ be a regular homomorphism. Choose a retraction $r\colon \Gamma(A)\longrightarrow \Gamma_{alg}(A)$ of the inclusion $\Gamma_{alg}(A)\subset \Gamma(A).$ Define $\psi_{un}$ as the composition
\[\psi_{un}\colon \Gamma(A)\buildrel r\oplus q\over\longrightarrow \Gamma_{alg}(A)\oplus(\Gamma(A)/\Gamma_{alg}(A))\buildrel \psi\oplus id\over\longrightarrow S(k)\oplus (\Gamma(A)/\Gamma_{alg}(A)),\]
where $q\colon \Gamma(A)\longrightarrow\Gamma(A)/\Gamma_{alg}(A)$ is the quotient map. Then, $\psi_{un}$ is an unpointed regular homomorphism. 
\end{lem}

\begin{proof}
First, notice that a retraction $r$ exists because $\Gamma_{alg}(A)$ is divisible by Proposition~\ref{prop: algebraic part is divisible}.
Now, we need to show that for an arbitrary $T\in\Sm/k$ and an arbitrary morphism  $Y\in\Hom_{\DM_{Nis}^{eff}(k)}(M(T),A),$ the composition 
\begin{displaymath}
\xymatrix{T(k) \ar[r]^-{Y_*} & \Gamma(A) \ar[r]^-{r\oplus q} \ar@/_20pt/[rr]_-{\psi_{un}} & \Gamma_{alg}(A)\oplus(\Gamma(A)/\Gamma_{alg}(A)) \ar[r]^-{\psi\oplus id} & S(k)\oplus (\Gamma(A)/\Gamma_{alg}(A))}
\end{displaymath}
is induced by a scheme morphism. Since it is enough to do this for each connected component of $T,$ we many assume that $T$ is connected. In this case, choosing a base point $t_0\in T(k),$ we can easily see that the above composition $\psi_{un}\circ Y_*$ is the translation of
\[T(k)\buildrel w_{Y}\over\longrightarrow \Gamma_{alg}(A)\buildrel\psi\over\longrightarrow S(k)\buildrel id\oplus 0\over\longrightarrow S(k)\oplus (\Gamma(A)/\Gamma_{alg}(A))\]
by the image of $t_0$ under $\psi_{un}\circ Y_*.$ Since $\psi$ is regular, $(id\oplus 0)\circ\psi\circ w_{Y}$ is induced by a scheme morphism, and hence, so is its translation $\psi_{un}\circ Y_*.$
\end{proof}

As usual, we say that an unpointed regular homomorphism $\phi\colon \Gamma(A)\longrightarrow G(k)$ is 
{\bf universal} if for any unpointed regular homomorphism $\phi'\colon \Gamma(A)\longrightarrow G'(k),$ there is a unique 
morphism of group schemes $h\colon G\longrightarrow G'$ such that $\phi'=h(k)\circ \phi.$

\begin{lem}\label{lem: pointed vs unpointed over closed fields}
Let $A\in \DM_{Nis}^{eff}(k).$ A universal regular homomorphism $\psi\colon \Gamma_{alg}(A)\longrightarrow S(k)$ exists if and only if a universal unpointed regular homomorphism $\phi\colon\Gamma(A)\longrightarrow G(k)$ exists. When they exist, $\phi_{alg}\colon \Gamma_{alg}(A)\longrightarrow G^0(k)$ (defined in the passage before Lemma~\ref{lem: construct unpointed regular from regular}) is a universal regular homomorphism for $\Gamma_{alg}(A).$
\end{lem}

\begin{proof}
For the ``if" part, let $\phi\colon\Gamma(A)\longrightarrow G(k)$ be a universal unpointed regular homomorphism. By Proposition~\ref{rem: surjection from a semiabelian variety}, the regular homomorphism $\phi_{alg}\colon \Gamma_{alg}(A)\longrightarrow G^0(k)$ gives rise to a surjective regular homomorphism $\phi_{alg}'\colon \Gamma_{alg}(A)\longrightarrow \mathrm{im}(\phi_{alg}).$ (We will eventually show $\mathrm{im}(\phi_{alg})=G^0(k).$) We claim that $\phi_{alg}'$ is universal.

Indeed, suppose $\psi'\colon \Gamma_{alg}(A)\longrightarrow S'(k)$ is an arbitrary regular homomorphism. $\psi'$ gives rise to an unpointed regular homomorphism $\psi'_{un}\colon \Gamma(A)\longrightarrow G'(k)$ by Lemma~\ref{lem: construct unpointed regular from regular}. By the universality of $\phi,$ there exists a unique homomorphism of groups schemes $h\colon G\longrightarrow G'$ with $\psi'_{un}=h(k)\circ\phi.$ Since $h$ can be restricted to a morphism between the identity components, it induces a homomorphism of semi-abelian varieties $h'\colon \mathrm{im}(\phi_{alg})\longrightarrow S',$ where $\mathrm{im}(\phi_{alg})$ denotes, by abuse of notation, the semi-abelian subvariety of $G^0$ whose group of $k$-points coincides with $\mathrm{im}(\phi_{alg}\colon\Gamma_{alg}(A)\longrightarrow G^0(k)).$ Now, by the construction, we have $\psi=h'(k)\circ\phi'_{alg}$ and $h'$ is the only morphism that satisfies this equality because $\phi'_{alg}$ is surjective and a morphism between smooth schemes is uniquely determined on the rational points as $k$ is algebraically closed (see Remark~\ref{rem: uniqueness}).

For the ``only if" part, suppose $\psi\colon \Gamma_{alg}(A)\longrightarrow S(k)$ is a universal regular homomorphism. By Lemma ~\ref{lem: construct unpointed regular from regular}, this gives rise to an unpointed regular homomorphism $\psi_{un}\colon \Gamma(A)\longrightarrow S(k)\oplus \Gamma(A)/\Gamma_{alg}(A).$ We claim that this is universal. 

Indeed, let $\phi'\colon\Gamma(A)\longrightarrow G'(k)$ be an arbitrary unpointed regular homomorphism. By the universality of $\psi,$ there is a unique homomorphism of group schemes $h\colon S\longrightarrow G'$ such that $h(k)\circ\psi=\phi'\circ\iota,$ where $\iota\colon \Gamma_{alg(A)}\longrightarrow \Gamma(A)$ is the inclusion. Now by the universality of pushout, there is a unique abelian group homomorphism $g\colon S(k)\oplus \Gamma(A)/\Gamma_{alg}(A)\longrightarrow G'(k)$ that makes the following diagram commutative:
\begin{displaymath}
\xymatrix{ \Gamma_{alg}(A) \ar[r]^-\psi \ar[d]_-\iota & S(k) \ar[d]_-{(id,0)} \ar@/^10pt/[ddr]^-{h(k)}\\
\Gamma(A) \ar[r]^-{\psi_{un}} \ar@/_10pt/[rrd]_-{\phi'} & S(k)\oplus \Gamma(A)/\Gamma_{alg}(A) \ar@{..>}[dr]^-{\exists!g} \\
 & & G'(k).}
\end{displaymath}
It now remains to show that $g$ is induced by a scheme morphism. But this is obvious because the diagram already shows that $g$ is induced by $h$ on the identity component.

For the final assertion, let $\phi\colon\Gamma(A)\longrightarrow G(k)$ be a universal unpointed regular homomorphism. We have shown that the restriction $\phi_{alg}'\colon\Gamma_{alg}(A)\longrightarrow \mathrm{im}(\phi_{alg})$ 
is a universal regular homomorphism, and this in turn gives rise to a universal unpointed regular homomorphism $(\phi_{alg}')_{un}$ whose target is $\mathrm{im}(\phi_{alg})\oplus \Gamma(A)/\Gamma_{alg}(A).$ Now, by the universality of the unpointed regular homomorphism, we conclude that $G^0(k)=\mathrm{im}(\phi_{alg}).$ Hence $\phi_{alg}=\phi_{alg}'.$
\end{proof}

\begin{rem}\label{rem: relation to ABV}
We thank the referee for pointing out the relation to the work of Ayoub and Barbieri-Viale (\cite{Ayoub-Barbieri-Viale}). Let $\mathrm{SAbS}/k$ be the category of semi-abelian schemes over $k.$ Universal unpointed regular homomorphisms can be considered as a partial left adjoint to the Yoneda embedding $\widetilde{(~)}\colon \mathrm{SAbS}/k\longrightarrow \HI_{Nis}(k)$ in the following sense. Let $\HI_{Nis}^{rep}(k)$ be the full subcategory of $\HI_{Nis}(k)$ consisting of sheaves $A\in \HI_{Nis}(k)$ such that $\Gamma(A)$ has a universal unpointed regular homomorphism, which we write as $\Gamma(A)\longrightarrow \widetilde{Alg(A)}(k)$ ($Alg(A)\in \mathrm{SAbS}/k$). By Lemma~\ref{lem: link}, the functor $Alg\colon \HI_{Nis}^{rep}(k)\longrightarrow \mathrm{SAbS}/k$ is left adjoint to $\widetilde{(~)}\colon \mathrm{SAbS}/k\longrightarrow \HI_{Nis}^{rep}(k).$

Thus, universal unpointed regular homomorphisms may be considered as a variant of the left adjoint functor $(~)^{\leq 1}\colon \HI_{tr}^{\acute et}(k)\longrightarrow \HI_{\leq 1}^{\acute et}(k)$ of the inclusion $\HI_{\leq 1}^{\acute et}(k)\subset \HI_{tr}^{\acute et}(k)$ defined in \cite[Definition 1.3.14]{Ayoub-Barbieri-Viale}. Here, $\HI_{tr}^{\acute et}(k)$ stands for the category of $\Z[1/p]$-coefficient homotopy invariant \'etale sheaves with transfers and $\HI_{\leq 1}^{\acute et}(k)$ is the smallest cocomplete Serre subcategory of the category of $\Z[1/p]$-coefficient \'etale sheaves with transfers that contains lattices (see \emph{[Ibid., the first sentence in Section 1.2]}) and sheaves represented by semi-abelian varieties. Notice that $\HI_{\leq 1}^{\acute et}(k)$ is cocomplete but $\mathrm{SAbS}/k$ consists of schemes locally of finite type.
\end{rem}


\subsection{Relation with Serre-Ramachandran's Albanese schemes}\label{subsection: Relation with Ramachandran's Albanese schemes}

Suppose $k$ is perfect. For any $X\in\Sm/k,$ there is a canonical map (\cite[the map (6)]{Spiess-Szamuely})
\begin{equation}
alb_X\colon \Z_{tr}(X)\longrightarrow \widetilde{Alb_X}
\end{equation}
of presheaves on $\Sm/k,$ where $\widetilde{Alb_X}$ is the sheaf represented by Serre-Ramachandran's Albanese scheme $Alb_X$ (\cite[Theorem 1.11]{Ramachandran}). This induces a morphism
\begin{equation}\label{map: sheaf albanese}
\tilde a_X\colon h_0^{Nis}(X)\longrightarrow \widetilde{Alb_X}
\end{equation}
in $\HI_{Nis}(k).$ (For details, see the discussion between Lemma 3.2 and Remark 3.3 in \cite{Spiess-Szamuely}.)

\begin{prop}\label{prop: Ramachandran-Serre}
Suppose $k$ is algebraically closed and $X\in\Sm/k.$ Let
\[a_X\colon \Gamma(h_0^{Nis}(X))= h_0^{Nis}(X)(k)\longrightarrow Alb_X(k),\]
be the section of $\tilde a_X$ over $k.$ Then, $a_X$ is a universal unpointed regular homomorphism for $\Gamma(h_0^{Nis}(X))$ and its restriction to the algebraic part 
\[\Gamma_{alg}(h_0^{Nis}(X))\longrightarrow Alb_X^0(k)\]
is a universal regular homomorphism.

Furthermore, via the canonical isomorphism $\Gamma(M(X))\longrightarrow \Gamma(h_0^{Nis}(X)),$ $a_X$ induces a universal unpointed regular homomorphism on $\Gamma(M(X)),$ and its restriction $\Gamma_{alg}(M(X))\longrightarrow Alb_X^0(k)$ is a universal regular homomorphism for $\Gamma_{alg}(M(X)).$
\end{prop}

\begin{proof}
By Lemma~\ref{lem: pointed vs unpointed over closed fields} and the second part of Lemma~\ref{lem: link}, it suffices to prove only the first statement that $a_X$ is a universal unpointed regular homomorphism for $\Gamma(h_0^{Nis}(X)).$ We thank the referee for providing the following proof.

By the first part of Lemma~\ref{lem: link}, it is enough to show that for any semi-abelian scheme $G,$ the pre-composition with $\tilde a_X$ 
\[\Hom_{\HI_{Nis}(k)}(\widetilde{Alb_X},\tilde G)\longrightarrow \Hom_{\HI_{Nis}(k)}(h_0^{Nis}(X),\tilde G)\]
is bijective. This map fits in a larger commutative diagram
\begin{displaymath}
\xymatrix{ \Hom_{\HI_{Nis}(k)}(\widetilde{Alb_X},\tilde G) \ar[r] & \Hom_{\HI_{Nis}(k)}(h_0^{Nis}(X),\tilde G) \ar[d]^-\sigma \\
\Hom_{\Sm/k}(X,G) \ar[u]^-\theta  \ar[r]^-\tau  & \Hom_{\Sh_{Nis}^{tr}(\Sm/k)}(\Z_{tr}(X),\tilde G)}
\end{displaymath}
where $\sigma,$ $\tau$ and $\theta$ are as follows. The map $\sigma$ is induced by the canonical map $q\colon\Z_{tr}(X)\longrightarrow h_0^{Nis}(X).$ For $a\in \Hom_{\Sm/k}(X,G),$ $\tau(a)\colon\Z_{tr}(X)\longrightarrow\tilde G$ is the morphism that sends $id_X\in\Z_{tr}(X)(X)$ to $a\in\tilde G(X).$ The map $\theta$ is given by the composition
\[\Hom_{\Sm/k}(X,G) \buildrel\cong\over\longleftarrow \Hom_{\mathrm{SAbS}/k}(Alb_X,G)\longrightarrow \Hom_{\HI_{Nis}(k)}(\widetilde{Alb_X},\tilde G),\]
where the first arrow is induced by the morphism $u_X\colon X\longrightarrow Alb_X$ that induces the morphism $\Z_{tr}(X)\buildrel alb_X\over\longrightarrow \widetilde{Alb_X}$ and the second by the Yoneda embedding $\widetilde{(-)}.$
The first map is an isomorphism by \cite[Theorem 1.11]{Ramachandran}.

Now, $\sigma$ is an isomorphism by the universality of $q$ (\cite[Example 2.20]{MVW}) and $\tau$ is also an isomorphism by definition. It is, therefore, enough to observe that $\theta$ is surjective. For this, let $f\in \Hom_{\HI_{Nis}(k)}(\widetilde{Alb_X},\tilde G).$ Evaluating $f$ at $u_X\in\widetilde{Alb_X}(X),$ we obtain $g\in \tilde G(X)=\Hom_{\Sm/k}(X,G),$ but then we clearly have $\theta(g)=f.$
\end{proof}

\begin{cor}\label{cor: existence in dim zero}
If $k$ is algebraically closed, then there exists a universal regular homomorphism for $\Gamma_{alg}(M(X))$ for any $X\in\Sch/k.$
\end{cor}

\begin{proof}
Since $M(X)=M(X_{red})$ by definition, we may assume that $X$ is reduced. Let $X_i$ ($i=1,\cdots,r$) be the irreducible components of $X.$ By de Jong's alteration \cite[Theorem 4.1]{de Jong}, there is a surjective morphism $X_i'\longrightarrow X_i$ from a smooth scheme $X_i'$ for each $r.$ The corollary follows once we show that the induced maps
\[\Gamma_{alg}\big(M(\coprod_{i=1}^r X'_i)\big)\buildrel f\over\longrightarrow \Gamma_{alg}\big(M(\coprod_{i=1}^r X_i)\big)\buildrel g\over\longrightarrow \Gamma_{alg}(M(X))\]
are surjective. Indeed, the surjectivity implies that any surjective regular homomorphism $\phi\colon \Gamma_{alg}(M(X))\longrightarrow S(k)$ is dominated by the universal regular homomorphism of $\Gamma_{alg}\big(M(\coprod_{i=1}^r X'_i)\big),$ which exists by Proposition~\ref{prop: Ramachandran-Serre}. By Corollary~\ref{cor: existence criterion with surjective regular homomorphisms}, we then conclude that a universal regular homomorphism exists for $\Gamma_{alg}(M(X)).$

To see the surjectivity of $f,$ notice that $\Gamma_{alg}(M(X_i))=H_0(X_i,Z)^0$ and $\Gamma_{alg}(M(X_i'))=H_0(X_i',Z)^0$ by Proposition~\ref{prop: algebraic part of zero cycles}. Since $k$ is algebraically closed, the surjectivity of $f$ follows from the surjectivity of the alterations $X_i'\longrightarrow X_i.$

For the map $g,$ consider the commutative square
\begin{displaymath}
\xymatrix{ \Gamma_{alg}\big(M(\coprod_{i=1}^r X_i)\big) \ar@{^{(}->}[r]^-h \ar[d]_-g & H_0(\coprod_{i=1}^r X_i,\Z)^0 \ar@{->>}[d] \\
\Gamma_{alg}(M(X)) \ar@{^{(}->}[r] & H_0(X,\Z)^0,}
\end{displaymath}
where the horizontal inclusions are those in Proposition~\ref{prop: algebraic part of zero cycles} and the vertical maps are canonical. The right vertical map is surjective because it is induced by the surjection $\coprod_{i=1}^rX_i\longrightarrow X.$

Now, it is straightforward to see that $\Gamma_{alg}\big(M(\coprod_{i=1}^r X_i)\big)$ is canonically isomorphic to $\bigoplus_{i=1}^r\Gamma_{alg}M(X_i),$ and the latter is canonically isomorphic to $\bigoplus_{i=1}^r H_0(X_i,\Z)^0$ by Proposition~\ref{prop: algebraic part of zero cycles}. Hence, the cokernel of $h$ is isomorphic to $\Z^{\oplus r-1}.$ Since $\Gamma_{alg}(M(X))$ is divisible by Proposition~\ref{prop: algebraic part is divisible}, this implies the surjectivity of $g$ as desired.
\end{proof}


\subsection{The existence of universal regular homomorphisms}

In this subsection, we derive the existence of universal regular homomorphisms from the finiteness of $\Z/l$-coefficient \'etale cohomology groups (Theorem~\ref{thm: existence main}). We begin by showing that $\DM_{Nis}^{eff}(k)$ satisfies Axiom~\ref{axiom}. We thank the referee for the following lemma and its proof.

\begin{lem}[{\cite{Lichtenbaum, Suslin-Voevodsky-Inventiones, VSF5}}]\label{lem: axiom}
Suppose $k$ is perfect. Let $C$ be a smooth connected affine curve pointed at a rational point $c\in C(k).$ Let $\iota_{c}\colon C\longrightarrow Alb_C^0$ be the canonical morphism that sends $c$ to $0.$ Then, the morphism $M(\iota_{c})\colon M(C)\longrightarrow M(Alb_C^0)$ in $\DM_{Nis}^{eff}(k)$ is a split monomorphism and has a multiplicative retraction.
\end{lem}

\begin{proof}
By \cite[Theorem 3.4.2]{VSF5}, there is a short exact sequence
\[0\longrightarrow\widetilde{Alb_C^0}\buildrel i\over\longrightarrow h_0^{Nis}(C)\longrightarrow \Z\longrightarrow0\]
of Nisnevich sheaves, where $i$ is the inverse of the isomorphism 
\[\mathrm{ker}\{h_0^{Nis}(C)\longrightarrow \Z\}\longrightarrow\widetilde{Alb_C^0}\]
that is induced by $\iota_{c};$ see also \cite[Lemma 3.2]{Spiess-Szamuely} and the succeeding paragraph there.

Since $C$ is affine, again by \cite[Theorem 3.4.2]{VSF5}, this short exact sequence gives rise to a distinguished triangle in $\DM_{Nis}^{eff}(k)$
\[\widetilde{Alb_C^0}\buildrel i'\over\longrightarrow M(C)\buildrel s\over\longrightarrow \Z\buildrel [+1]\over\longrightarrow,\]
where $s$ is induced by the structure morphism of $C.$

Now, define a morphism $\rho\in\Hom_{\DM_{Nis}^{eff}(k)}(M(Alb_C^0), M(C))$ as the composition
\[M(Alb_C^0)\buildrel (p,q)\over\longrightarrow \widetilde{Alb_C^0}\oplus\Z\buildrel i'\oplus M(c)\over\longrightarrow M(C),\]
where $p$ and $q$ are the canonical morphisms. (Notice that $i'\oplus M(c)$ is an isomorphism. We shall use this fact later in the proof.) We claim that $\rho$ is a multiplicative retraction of $M(\iota_{c}).$

For the multiplicativity, let 
\[M(p_1)+M(p_2)-M(m)-M(0)\colon M(Alb_C^0\times Alb_C^0)\longrightarrow M(Alb_C^0)\]
be the morphism as in Definition~\ref{defn: multiplicative}. Then, the morphism 
\[p\circ (M(p_1)+M(p_2)-M(m)-M(0))\colon M(Alb_C^0\times Alb_C^0)\longrightarrow\widetilde{Alb_C^0}\]
corresponds to $p_1+p_2-m-0=0\in \widetilde{Alb_C^0}(Alb_C^0\times Alb_C^0)$ under the identification $\Hom_{\DM_{Nis}^{eff}(k)}(M(Alb_C^0\times Alb_C^0),\widetilde{Alb_C^0})\cong \widetilde{Alb_C^0}(Alb_C^0\times Alb_C^0)$ (the isomorphism holds because $\widetilde{Alb_C^0}$ is homotopy invariant and hence $\A^1$-local by \cite[Proposition 14.8]{MVW}), and the morphism 
\[q\circ (M(p_1)+M(p_2)-M(m)-M(0))\colon M(Alb_C^0\times Alb_C^0)\longrightarrow\Z\]
corresponds to $str+str-str-str=0\in\Z_{tr}(k)(Alb_C^0\times Alb_C^0)$ under a similar identification. Here, $str$ stands for the structure morphism $Alb_C^0\times Alb_C^0\longrightarrow \Spec k.$ Therefore, $\rho\circ(M(p_1)+M(p_2)-M(m)-M(0))=(i'\oplus M(c))\circ0=0.$

To show that $\rho$ is indeed a retraction of $M(\iota_{c}),$ we need to prove 
\[(i'\oplus M(c))\circ(p,q)\circ M(\iota_c)= id_{M(C)}.\]
Since $i'\oplus M(c)$ is an isomorphism, this is equivalent to
\[(p,q)\circ M(\iota_c)\circ (i'\oplus M(c))= id_{\widetilde{Alb_C^0}\oplus\Z}.\]
Since it is clear that $q\circ M(\iota_c)\circ M(c)=id_\Z,$ it suffices to show that $p\circ  M(\iota_{c})\circ i' = id_{\widetilde{Alb_C^0}}$ holds. We prove this in the next lemma.
\end{proof}

\begin{lem}
With the same notation as in the previous proof, the following equality holds:
\[ p\circ  M(\iota_{c})\circ i' = id_{\widetilde{Alb_C^0}}.\]
\end{lem}

\begin{proof}
We keep using the symbols introduced in the previous proof. The proof is about unwinding the definition of $i'.$  Let us begin with the following commutative diagram in $\Sh_{Nis}^{tr}(k)$
\begin{displaymath}
\xymatrix{& \Z_{tr}(C) \ar[d]^-l \ar[dl]_-\kappa\\
\widetilde{Alb_C^0} & h_0^{Nis}(C)\ar[r]^-a \ar[l]_-{\bar\kappa}& \Z\\
& \ker(a) \ar@{^(->}[u]_k \ar[ul]^-\cong_-{\bar\kappa\circ k}}
\end{displaymath}
where $l$ is the quotient map, $\kappa$ is the map induced by $\iota_c$ and $\bar\kappa$ by $\kappa$ (by the universality of $l$; see \cite[Example 2.20]{MVW}), $a$ is canonical and $k$ is the inclusion. $\bar\kappa\circ k$ is an isomorphism by \cite[Theorem 3.4.2]{VSF5}. By definition, we have $i=k\circ (\bar\kappa\circ k)^{-1}.$

Consider the image of this diagram under the localization functor
\[L\colon D(\Sh_{Nis}^{tr}(k))\longrightarrow \DM_{Nis}^{eff}(k).\]
$L(l)$ is an isomorphism by \cite[Theorem 3.4.2]{VSF5} and we have by definition $i' = L(l)^{-1}\circ L(i).$ Also, it is immediate from definitions that $L(\kappa) = p\circ M(\iota_c).$ Hence, we have
\begin{eqnarray}
p\circ  M(\iota_{c})\circ i'&=& L(\kappa) \circ L(l)^{-1}\circ L(i) = L(\bar\kappa)\circ L(i)\nonumber\\ 
&=& L(\bar\kappa)\circ L(k)\circ L((\bar\kappa\circ k)^{-1})=L(\bar\kappa\circ k)\circ L((\bar\kappa\circ k)^{-1}) = id.\nonumber
\end{eqnarray}
\end{proof}

\begin{rem}\label{Weil's theorem}
From what we have proved and the $\mathbb A^1$-homotopy invariance of motives, we can extract the following fact on algebraic equivalence on Chow groups due to Weil (\cite[Lemma 9]{Weil}; see also \cite[p.60, Theorem 1]{Lang}):

Suppose $k$ is algebraically closed. Then, for any smooth 
proper connected scheme $X,$ the equality
\[A^r(X) = \bigcup_{\substack{ A\in\mathfrak A\\ a_0,a_1\in A(k)}}  \mathrm{im}\{CH^r(A\times X)\buildrel a_1^*-a_0^*\over\longrightarrow CH^r(X)\}\]
holds, where $A^r(X)$ is as in Subsection~\ref{section: The case of Chow groups} and $\mathfrak A$ is the class of abelian varieties.

Indeed, let $\mathbb M:=\underline{\Hom}(M(X),\Z(r)[2r]).$ By Proposition~\ref{prop: parametrization by semiabelian varieties} (which is applicable by Lemma~\ref{lem: axiom}) and Proposition~\ref{prop: comparison with the classical algebraic part}, it is enough to observe the equality (cf. the proof of Proposition~\ref{prop: algebraic part is divisible})
\begin{eqnarray}
&&\bigcup_{S\in\mathfrak S}\mathrm{im}\{H_0(S,\Z)^0 \times \Hom_{\DM_{Nis}^{eff}(k)}(M(S),\mathbb M)\buildrel P_S\over\longrightarrow \Gamma(\mathbb M)\}\nonumber\\
&=&\bigcup_{A\in\mathfrak A}\mathrm{im}\{H_0(A,\Z)^0 \times \Hom_{\DM_{Nis}^{eff}(k)}(M(A),\mathbb M)\buildrel P_A\over\longrightarrow \Gamma(\mathbb M)\},\nonumber
\end{eqnarray}
where $\mathfrak S$ is the class of semi-abelian varieties and $P_S$ and $P_A$ are defined by composition in $\DM_{Nis}^{eff}(k).$ By induction on the torus rank of $S,$ it suffices to show that for any exact sequence $0\to\mathbb G_m\to S\to S'\to0$ ($S$ and $S'$ are semi-abelian) there is an inclusion
\begin{eqnarray}
&&\mathrm{im}\{H_0(S,\Z)^0 \times \Hom_{\DM_{Nis}^{eff}(k)}(M(S),\mathbb M)\buildrel P_S\over\longrightarrow \Gamma(\mathbb M)\}\nonumber\\
&\subset&\mathrm{im}\{H_0(S',\Z)^0 \times \Hom_{\DM_{Nis}^{eff}(k)}(M(S'),\mathbb M)\buildrel P_{S'}\over\longrightarrow \Gamma(\mathbb M)\}.\nonumber
\end{eqnarray}

Since $S$ is a $\mathbb G_m$-torsor over $S'$ in the Zariski topology, there is an associated line bundle $p\colon E\longrightarrow S'$ with the
zero section $s\colon S'\longrightarrow E$ such that $E\setminus s(S')\cong S.$ Regarding $S$ as a subscheme of $E$ via this isomorphism, consider the commutative diagram
\begin{displaymath}
\xymatrix{ H_0(S,\Z)^0 \ar[d]_{inc_*}&  \times & \Hom_{\DM_{Nis}^{eff}(k)}(M(S),\mathbb M)  & \buildrel P_S\over\longrightarrow & \Gamma(\mathbb M) \ar@{=}[d] \\
H_0(E,\Z)^0 & \times & \Hom_{\DM_{Nis}^{eff}(k)}(M(E),\mathbb M) \ar@{->>}[u]^-{inc^*} &\buildrel P_E\over\longrightarrow & \Gamma(\mathbb M)}
\end{displaymath}
The map $inc^*$ is surjective because it can be identified with the pullback of Chow groups $CH^r(E\times X)\longrightarrow CH^r(S\times X)$ 
along the open immersion $S\times X\hookrightarrow E\times X.$ Thus, the commutativity of the diagram implies 
$\mathrm{im}(P_S)\subset\mathrm{im}(P_{E}).$ Finally, by the $\A^1$-homotopy invariance, $p$ induces an isomorphism 
of motives $M(E)\buildrel\cong\over\longrightarrow M(S').$ Hence, we obtain  $\mathrm{im}(P_S)\subset\mathrm{im}(P_{E})=\mathrm{im}(P_{S'}).$
\end{rem}

\begin{thm}\label{thm: existence main}
Suppose $k$ is algebraically closed. Let $X\in \Sch/k$ and let $n$ be a non-negative integer and $m$ be an integer. If $m\leq n+2,$ then there exists a universal regular homomorphism for $\Gamma_{alg}(\underline{\Hom}(M^c(X),\Z(n)[m])).$
\end{thm}

\begin{proof}
Let $\Z/l(n)$ denote the Suslin-Voevodsky motivic complex $\Z(n)$ tensored with $\Z/l.$ 
By Corollary~\ref{cor: practical criterion}, it suffices to show that $_l\Gamma(\underline{\Hom}(M^c(X),\Z(n)[m]))$ is finite for some prime $l$ that is different from $\mathrm{char}~k.$ Since this group received a surjection from $\Gamma(\underline{\Hom}(M^c(X),\Z/l(n)[m-1])),$ it is enough to show:

\skipline
\emph{Claim. $\Gamma(\underline{\Hom}(M^c(X),\Z/l(n)[m]))$ is finite for any $X\in\Sch/k$ if $m\leq n+1$ and $l\neq \mathrm{char}~k.$}
\skipline

Let us first prove this for smooth proper schemes $X.$ In this case, we have $M^c(X)=M(X),$ so there is an isomorphism 
\[\Gamma(\underline{\Hom}(M^c(X),\Z/l(n)[m]))\cong H_{Nis}^{m}(X,\Z/l(n))\]
by \cite[Corollary 2.19]{MVW} and \cite[Example 11.1.3]{Cisinski-Deglise}.

By the Beilinson-Lichtenbaum-Rost-Voevodsky theorem (see \cite[Introduction, Theorem C]{HW-norm-residue}), the canonical map
\[H_{Nis}^{m}(X,\Z/l(n))\longrightarrow H_{\acute et}^{m}(X,\mu_l^{\otimes n})\]
is an isomorphism if $m < n+1$ and still an injection for $m = n+1.$ The claim therefore follows from the finiteness of the \'etale cohomology $H_{\acute et}^m(X,\mu_l^{\otimes n})$ (see \cite[Chapter VI, Corollary 2.8]{Milne etale}).

Next, we prove the general case by the induction on dimension. Let $p$ be the exponential characteristic of $k.$ Since $l$ is prime to $p,$ we have
\[\Gamma(\underline{\Hom}(M^c(X),\Z/l(n)[m]))\cong \Hom_{\DM_{Nis}^{eff}(k,\Z[\frac{1}{p}])}(\Z[\frac{1}{p}], \underline{\Hom}(M^c(X)[\frac{1}{p}],\Z/l(n)[m])),\]
where $\DM_{Nis}^{eff}(k,\Z[\frac{1}{p}])$ is the derived category of effective motives with $\Z[\frac{1}{p}]$-coefficients in \cite[Definition 11.1.1]{Cisinski-Deglise} and \cite[Definition 4.0.5]{Kelly}.

For an arbitrary $X\in \Sch/k,$ there is a pullback diagram by \cite[Expos\'e 0, Theorem 3 (1)]{ILO}
\begin{displaymath}
\xymatrix{U' \ar[d]_-{f'} \ar@{^{(}->}[rr] & & \bar X' \ar[d]^-f \\
U \ar@{^{(}->}[r] & X \ar@{^{(}->}^-i[r] & \bar X}
\end{displaymath} 
where $i$ is a Nagata compactification, $f$ is a Gabber's $l'$-alteration with $\bar X'$ smooth and projective, and $U$ is a smooth open dense subscheme of $X$ such that $f'$ is \'etale and finite of degree prime to $l.$ 

By the localization triangle in $\DM_{Nis}^{eff}(k,\Z[\frac{1}{p}])$ (\cite[Proposition 5.3.5]{Kelly}; cf. \cite[Proposition 4.1.5]{VSF5} under resolution of singularities with integral coefficients)
\[M^c(\bar X'\setminus U')[\frac{1}{p}]\longrightarrow M^c(\bar X')[\frac{1}{p}]\longrightarrow M^c(U')[\frac{1}{p}]\buildrel [+1]\over \longrightarrow,\]
the claim for $U'$ follows from the smooth proper case of $\bar X'$ and the induction hypothesis. 

Now, since $f'$ is finite and \'etale, 
the composition 
\[M^c(U)[\frac{1}{p}]\buildrel f'^*\over\longrightarrow M^c(U')[\frac{1}{p}]\buildrel f'_*\over\longrightarrow M^c(U)[\frac{1}{p}]\]
is the multiplication by the degree of $f',$ where $f'^*$ (resp. $f'_*$) is defined by the flat pullback (resp. proper pushforward) of cycles. As the degree of $f'$ is prime to $l,$ the composition of the induced maps
\[\underline{\Hom}(M^c(U)[\frac{1}{p}],\Z/l(n)[m])\longrightarrow \underline{\Hom}(M^c(U')[\frac{1}{p}],\Z/l(n)[m])\]
\[\longrightarrow \underline{\Hom}(M^c(U)[\frac{1}{p}],\Z/l(n)[m])\]
is an isomorphism. Hence, the claim for $U$ follows from that for $U'.$

Finally, consider the localization triangle associated with $U\hookrightarrow X.$ Since $\dim X\setminus U <\dim X,$ the claim for $X$ follows from the induction hypothesis and the claim for $U.$
\end{proof}


\subsection{Rationality}\label{subsection: rationality}

Let $k$ be a perfect field and let $\bar k$ be an algebraic closure. In this subsection, we assume $X$ to be a smooth proper $k$-scheme and write $X_{\bar k}:=X\times_k\bar k.$ Since $X$ is proper, we have $M^c(X_{\bar k})=M(X_{\bar k}).$ We are going to consider the rationality of the universal regular homomorphism for $\Gamma_{alg}(\underline{\Hom}(M(X_{\bar k}),\Z(n)[m])).$ The assumption that $X$ is smooth and proper is made because we will need to interpret the group $\Hom_{\DM_{Nis}^{eff}(k)}(M(T),\underline{\Hom}(M(X_{\bar k}),\Z(n)[m]))$ ($T\in\Sm/\bar k$) geometrically as the higher Chow group $CH^n(T\times X_{\bar k}, 2n-m).$

We write $\tilde\sigma$ for the scheme morphism $\Spec \sigma\colon\Spec\bar k\longrightarrow \Spec\bar k$ induced by an automorphism $\sigma\in \Gal(\bar k/k).$ For $\bar k$-schemes $X_1$ and $X_2,$ a scheme morphism $f\colon X_1\longrightarrow X_2$ is said to be {\bf over $\tilde\sigma$} if the diagram
\begin{displaymath}
\xymatrix{ X_1 \ar[r]^-f \ar[d]  & X_2 \ar[d]\\
\Spec \bar k \ar[r]^-{\tilde\sigma} & \Spec \bar k}
\end{displaymath}
commutes. Here the vertical arrows are the structure morphisms. For a morphism $f$ over $\tilde\sigma,$ we define a map $f(\bar k/\tilde\sigma)\colon X_1(\bar k)\longrightarrow X_2(\bar k)$ by $f(\bar k/\tilde\sigma)(x)=f\circ x\circ\tilde\sigma^{-1}.$

The proof of the following theorem is a direct translation of the proof of \cite[Theorem 4.4]{ACMV} into our setting. We thank the referee for suggesting the second statement in the theorem. 

\begin{thm}\label{thm: rationality}
Suppose $\Gamma_{alg}(\underline{\Hom}(M(X_{\bar k}),\Z(n)[m]))$ has a universal regular homomorphism (for example, in the situation of Theorem~\ref{thm: existence main})
\[\Phi\colon \Gamma_{alg}(\underline{\Hom}(M(X_{\bar k}),\Z(n)[m]))\longrightarrow S(\bar k).\]
Then, the semi-abelian $\bar k$-variety $S$ has a model over $k,$ i.e. there is a semi-abelian $k$-variety $\underline S$ with $S\cong\underline S\times_k\bar k.$

In addition, $\Phi$ also descends in the following sense: the composition
\[\Gamma_{alg}(\underline{\Hom}(M(X),\Z(n)[m]))\longrightarrow\Gamma_{alg}(\underline{\Hom}(M(X_{\bar k}),\Z(n)[m]))\buildrel\Phi\over\longrightarrow S(\bar k),\]
where the first arrow is the canonical map, factors through $\underline S(k)\to \underline S(\bar k)\cong S(\bar k),$ the map defined by the choice of a model $\underline S.$
\end{thm}

\begin{proof}
Let us first prove the existence of $\underline S.$
As explained in \cite[Section 4.1 and Remark 4.1]{ACMV} together with \cite[Lemma 4.1]{Kahn} (or more classically \cite[Theorem 3]{Rosenlicht} or \cite[Theorem 2]{Iitaka}), it suffices to construct a $\bar k/k$-descent datum on $S$ that 
fixes $0$ of $S,$ i.e. a system of scheme isomorphisms $h_{\tilde\sigma}\colon S\buildrel\cong\over\longrightarrow S$ over $\tilde\sigma$ ($\sigma\in\Gal(\bar k/k)$) such that $h_{\tilde\sigma}(0)=0$ and $h_{\tilde\sigma\circ\tilde\tau}=h_{\tilde\sigma}\circ h_{\tilde\tau}$ for any $\sigma,\tau\in \Gal(\bar k/k).$

By Lemma~\ref{lem: pointed vs unpointed over closed fields}, there is an universal unpointed regular homomorphism
\[\Psi\colon \Gamma(\underline{\Hom}(M(X_{\bar k}),\Z(n)[m]))\longrightarrow G(\bar k).\]
Observe that it suffices to construct a system of scheme isomorphisms $k_{\tilde\sigma}\colon G\buildrel\cong\over\longrightarrow G$ over $\tilde\sigma$ with $k_{\tilde\sigma}(0)=0$ and $k_{\tilde\sigma\circ\tilde\tau}=k_{\tilde\sigma}\circ k_{\tilde\tau}$ ($\sigma, \tau\in \Gal(\bar k/k)$) because $k_{\tilde\sigma}$'s restrict to the identity component $G^0$ of $G$ (which is isomorphic to $S$) and give rise to a desired descent datum $\{h_{\tilde\sigma}\}$ on $S.$

Via the comparison isomorphism in \cite[Theorem 19.1]{MVW}, $\Psi$ can be regarded as an initial homomorphism 
\[\psi\colon CH^n(X_{\bar k}, 2n-m)\longrightarrow G(\bar k)\]
among homomorphisms $\psi'\colon CH^n(X_{\bar k}, 2n-m)\longrightarrow G'(\bar k)$ ($G'$ is an arbitrary semi-abelian $\bar k$-scheme) such that for any $T\in\Sm/\bar k$ and $Y\in CH^n(T\times X_{\bar k}, 2n-m),$ the composition 
\[T(\bar k) \buildrel Y_*\over\longrightarrow CH^n(X_{\bar k}, 2n-m)\buildrel\psi'\over\longrightarrow G'(\bar k)\]
is induced by a morphism of $\bar k$-schemes $T\longrightarrow G'.$ Here, $Y_*$ sends a $\bar k$-point $t\in T(\bar k)$ to the pullback of $Y$ along the morphism $t\times id\colon\Spec\bar k\times X_{\bar k}\longrightarrow T\times X_{\bar k}.$ (For the pullback of higher Chow groups of smooth schemes, see \cite[Theorem 17.21]{MVW}.) Let us call such a homomorphism $\psi'$ regular too because it is just an unpointed regular homomorphism formulated in terms of higher Chow groups.

Now, we come to the key step of the proof. For $\sigma\in\Gal(\bar k/k),$ consider the induced isomorphism $\sigma_{X_{\bar k}}:=id_X\times \tilde\sigma\colon X_{\bar k}\longrightarrow X_{\bar k},$ and let $\sigma_{X_{\bar k}}^*\colon CH^n(X_{\bar k},2n-m)\longrightarrow CH^n(X_{\bar k},2n-m)$ be the pullback. We claim that the composition 
\[T(\bar k)\buildrel Y_*\over\longrightarrow CH^n(X_{\bar k},2n-m)\buildrel \sigma_{X_{\bar k}}^*\over\longrightarrow CH^n(X_{\bar k},2n-m) \buildrel\psi\over\longrightarrow G(\bar k),\]
where $T,$ $Y$ and $\psi$ are as above, is induced by a morphism $f\colon T\longrightarrow G$ over $\tilde\sigma^{-1},$ i.e. $f(\bar k/\tilde\sigma^{-1}) = \psi\circ\sigma_{X_{\bar k}}^*\circ Y_*.$ (Note that we have $\tilde\sigma^{-1}=\widetilde{\sigma^{-1}}.$) Indeed, let $\sigma_T\colon T^\sigma\longrightarrow T$ be the base change of $\tilde\sigma$ along the structure morphism $T\longrightarrow \bar k.$ Consider the commutative diagram
\begin{displaymath}
\xymatrixcolsep{4pc}\xymatrix{ T(\bar k)  \ar[r]^-{Y_*} \ar[d]_-{\sigma_T^{-1}(\bar k/\tilde\sigma^{-1})} & CH^n(X_{\bar k}, 2n-m)  \ar[d]^-{\sigma_{X_{\bar k}}^*} \\
T^\sigma(\bar k)  \ar[r]_-{(\sigma_T\times_{\tilde\sigma}\sigma_{X_{\bar k}})^*(Y)_*} & CH^n(X_{\bar k}, 2n-m) \ar[r]^-{\psi} & G(\bar k).}
\end{displaymath}
Let us write $t^\sigma = \sigma_T^{-1}(\bar k/\tilde\sigma^{-1})(t)$ for $t\in T(\bar k).$ The commutativity of the above diagram is a result of the commutativity of the following square 
\begin{displaymath}
\xymatrix{ X_{\bar k} \ar[r]^-{t^\sigma\times id} \ar[d]_-{\sigma_{X_{\bar k}}} & T^\sigma \times_{\bar k} X_{\bar k} \ar[d]^-{\sigma_T\times_{\tilde\sigma} \sigma_{X_{\bar k}}}\\
 X_{\bar k} \ar[r]_-{t\times id} & T \times_{\bar k} X_{\bar k}.}
\end{displaymath}
Now, since $\psi\circ (\sigma_T\times_{\tilde\sigma}\sigma_{X_{\bar k}})^*(Y)_*$ is induced by a $\bar k$-morphism $T^\sigma\longrightarrow G$ by the regularity of $\psi,$ the composition $\psi\circ\sigma_{X_{\bar k}}^*\circ Y_* = \psi\circ (\sigma_T\times_{\tilde\sigma}\sigma_{X_{\bar k}})^*(Y)_*\circ \sigma_T^{-1}(\bar k/\tilde\sigma^{-1})$ is induced by a morphism $T\longrightarrow G$ over $\tilde\sigma^{-1}.$

The claim we have just proved implies that the composition 
\[ CH^n(X_{\bar k},2n-m)\buildrel \sigma_{X_{\bar k}}^*\over\longrightarrow CH^n(X_{\bar k},2n-m) \buildrel\psi\over\longrightarrow G(\bar k) \buildrel (\sigma^{-1})_{G}^{-1}(\bar k/\tilde\sigma) \over\longrightarrow G^{\sigma^{-1}}(\bar k)\]
is regular. (The last map is by definition the map induced by the inverse of $(\sigma^{-1})_{G}\colon G^{\sigma^{-1}}\longrightarrow G.$) Now, by the universality of $\psi,$ there exists a unique $\bar k$-morphism $l_\sigma\colon G\longrightarrow G^{-\sigma}$ that fits in the diagram
\begin{displaymath}
\xymatrixcolsep{4pc}\xymatrix{CH^n(X_{\bar k}, 2n-m)  \ar[d]^-{\sigma_{X_{\bar k}}^*} \ar[r]^-{\psi} & G(\bar k)  \ar@{..>}[dr]^-{\exists!~l_\sigma(\bar k)}\\
CH^n(X_{\bar k}, 2n-m) \ar[r]^-{\psi} & G(\bar k) \ar[r]_-{(\sigma^{-1})_G^{-1}(\bar k/\tilde\sigma)} & G^{\sigma^{-1}}(\bar k).}
\end{displaymath}

Let us put $m_\sigma :=(\sigma^{-1})_G\circ l_\sigma\colon G\longrightarrow G.$ It is a morphism over $\tilde\sigma^{-1}.$ Since $\psi$ and $\sigma_{X_{\bar k}}^*$ are group homomorphisms, $m_\sigma$ sends $0$ to $0.$ The uniqueness of $l_\sigma$ implies $m_{\tau\sigma}=m_\tau\circ m_\sigma$ for any $\sigma,\tau\in\Gal(\bar k/k).$ In particular, $m_\sigma$ is invertible. Now, define $k_{\tilde\sigma} := (m_\sigma)^{-1}.$ It is clear that $\{k_{\tilde\sigma}\}$ is the desired system of isomorphisms on $G.$

Finally, for the descent property of $\Phi,$ it suffices to show that $\Gal(\bar k /k)$ acts trivially on the image of 
\[\Gamma_{alg}(\underline{\Hom}(M(X),\Z(n)[m]))\longrightarrow\Gamma_{alg}(\underline{\Hom}(M(X_{\bar k}),\Z(n)[m]))\buildrel\Phi\over\longrightarrow S(\bar k),\]
where $\Gal(\bar k /k)$ acts on $S(\bar k)$ on the right via $\{h_{\tilde\sigma}\}_{\sigma\in\Gal(\bar k/k)}.$
Since $h_{\tilde\sigma}$ is defined as the restriction of $k_{\tilde\sigma}$ to the identity component, it is enough to observe that $\Gal(\bar k/k)$ acts (via $k_{\tilde\sigma}$) trivially on the image of
\[CH^n(X, 2n-m) \longrightarrow CH^n(X_{\bar k}, 2n-m) \buildrel\psi\over\longrightarrow G(\bar k).\]
This indeed follows from the following diagram, which commutes by the definition of $k_{\tilde\sigma}$:
\begin{displaymath}
\xymatrixcolsep{4pc}\xymatrix{CH^n(X, 2n-m) \ar[r] \ar[d]^-{id_X^*} & CH^n(X_{\bar k}, 2n-m)  \ar[d]^-{\sigma_{X_{\bar k}}^*} \ar[r]^-{\psi} & G(\bar k) \ar[d]^-{k_{\tilde\sigma}(\bar k)^{-1}} \\
CH^n(X, 2n-m) \ar[r] & CH^n(X_{\bar k}, 2n-m) \ar[r]^-{\psi} & G(\bar k).}
\end{displaymath}
\end{proof}

\appendix

\section{Universal regular homomorphisms for \'etale motives}\label{app}

\hfill by Bruno Kahn
\bigskip

In this appendix, we show that universal regular homomorphisms exist for all geometric \'etale motives and compare them with the ones for Nisnevich motives when they exist.

Let $\sT$ be a triangulated category. We say that a set $\sX$ of objects \emph{generates} $\sT$ if the smallest triangulated subcategory of $\sT$ containing $\sX$ and stable under direct summands is equal to $\sT$. We write $p$ for the exponential characteristic of $k$.

\begin{lemma}\label{la1} Let $DM_{\gm,\et}^\eff(k)$ be the category of {\rm [B-VK16, Def. 2.1.1]} (a $\Z[1/p]$-linear triangulated category), and let $DM_{\gm,\et}(k)$ be the category obtained from $DM_{\gm,\et}^\eff(k)$ by $\otimes$-inverting the Tate object. Then\\
a) the functor $DM_{\gm,\et}^\eff(k)\to DM_{\gm,\et}(k)$ is fully faithful;\\
b) $DM_{\gm,\et}(k)$ is generated by the $M(X)(-n)$ for $X$ smooth projective and $n\in \N$, and this $\otimes$-category is rigid.
\end{lemma}

\begin{proof} For a), use \cite[Prop. A.3]{motiftate}. By Gabber's refinement of de Jong's theorem \cite[Exp. X, th. 2.1]{gabber}, $DM_{\gm,\et}^\eff(k)$ is generated by the $M(X)$ for $X$ smooth projective; the first claim of b) follows. Thus does the second one, since each $M(X)$ is strongly dualisable \cite[App. B, ii)]{motiftate}. 
\end{proof}

For $A\in \sT$ and $n>0$, we write $A/n$ for the cone of $A\by{n} A$. This is well-defined up to non-unique isomorphism. For simplicity, assume that $\sT$ has a tensor structure with unit object $\Z$; we may choose a $\Z/n$ once and for all and define $A/n:=A\otimes \Z/n$ for any $A$, making $A\mapsto A/n$ a triangulated functor. 

\begin{thm}\label{ta1} Let $k$ be algebraically closed, and let $l$ be a prime number different from $p$. Then, for any $A\in DM_{\gm,\et}(k)$,\\
a) $\Gamma(A/l)$ is finite;\\
b)  ${}_l\Gamma(A)$ is finite, hence also ${}_l\Gamma_\alg(A)$;\\
c) a universal regular homomorphism exists for $\Gamma_\alg(A)$.
\end{thm}

\begin{proof} For simplicity, write $\sT:=DM_{\gm,\et}(k)$. Let $\sA$ be the full subcategory of $\sT$ consisting of those $A$'s such that $A(n)[i]$ verifies a) for any $(n,i)\in \Z\times \Z$: we have to show that $\sA=\sT$.

Let $A'\to A\to A''\by{+1}$ be an exact triangle in $\sT$. If $\Gamma(A'/l)$ and $\Gamma(A''/l)$ are finite, so is $\Gamma(A/l)$. Since $\sA$ is stable under translation, it is a triangulated subcategory of $\sT$. Since it is also clearly stable under direct summands, it suffices by Lemma \ref{la1} b) to show that $M(X)(-n)\in \sA$ for any smooth projective $X$ and any $n\in \N$. 

If $X$ is smooth projective of dimension $d$, Poincar\'e duality gives an isomorphism in $DM_\gm(k)$, hence also in $DM_{\gm,\et}(k)$:
\[M(X)\simeq M(X)^*(d)[2d]\]
where $M(X)^*$ is the dual of $M(X)$. Then, for $(n,i)\in \N\times \Z$, we get
\begin{multline*}
\Gamma(M(X)(n)[i]/l)\simeq \sT(M(X),\Z/l(d+n)[2d+i])\\
\simeq \sT(M(X)(r)[2r],\Z/l(d+r+n)[2(d+r)+i])
\end{multline*}
for any $r\in\Z$. Choose $r\ge 0$ such that $n+r\ge 0$. Up to replacing $X$ by $X\times \P^r$, we may assume $r=0$. By Lemma \ref{la1} a), the right hand side is then isomorphic to
\[H^{2d+i}_{M,\et}(X,\Z/l(d+n)) \text{ (\'etale motivic cohomology)}.\]

But this is ordinary \'etale cohomology by [SV00], so this group is (classically) finite.

b) follows from a) by the exact sequence
\[\Gamma(A/l[-1])\to \Gamma(A)\by{l} \Gamma(A).\]

Finally,  b) $\Rightarrow$ c) follows from Corollary 2.20.
\end{proof}

We have the following lemmas concerning the general situation of \S 2.2, where $\sA$ is just a preadditive category. 

\begin{lemma}\label{la3} Assume that $\Gamma_\alg(M(C))$ is divisible for any smooth curve $C$. Let $f:A\to B$ be a morphism in $\sA$.  Suppose that $\Coker(\sA(M(C),A)\allowbreak\by{f_*} \sA(M(C),B))$ is torsion for any $C$. Then $\Gamma_\alg(A)\to \Gamma_\alg(B)$ is surjective. If $A$ admits a universal regular homomorphism, so does $B$, and the corresponding homomorphism of universal semi-abelian varieties is faithfully flat.
\end{lemma}

\begin{proof} That of the first statement is a variant of the argument in the proof of Proposition 3.1, using the fact that the tensor product of a divisible group and a torsion group is $0$. Any regular homomorphism for $B$ induces one for $A$, hence the second statement follows from Corollary 2.19 and Proposition 2.15.
\end{proof} 

\begin{rem} One may wonder about a general version of Ro\v\i tman's theorem, i.e. ask when the universal map from Proposition 2.15 is bijective on $l$-torsion. At any rate, Lemma \ref{la3} shows that this is false most of the time: if $A=0$ in this lemma, one gets that the universal regular homomorphism for $B$ is $0$. In $DM_\gm(k)$ or $DM_{\gm,\et}(k)$, one may take  $B=C/n$ for any $C$ and any $n>0$ as an example. Finding good conditions for such a ``Ro\v\i tman theorem'' (beyond the known cases) looks like a hard geometric question.
\end{rem}

\begin{lemma}\label{la4} Let $T:\sA\to \sB$ be an additive functor to another preadditive category $\sB$. Assume that, with obvious notation, $\Gamma_\alg^\sA(M(C))$ and $\Gamma_\alg^\sB(TM(C))$ are divisible for any smooth curve $C$. Let $A\in \sA$. Suppose that $\Coker(\sA(M(C),A)\by{T} \sB(TM(C),T(A)))$ is torsion for any $C$. Then $T:\Gamma_\alg^\sA(A)\to \Gamma_\alg^\sB(T(A))$ is surjective. If $A$ admits a universal regular homomorphism, so does $T(A)$, and the corresponding homomorphism of universal semi-abelian varieties is faithfully flat.
\end{lemma}

The proof is completely parallel to the one of Lemma \ref{la3}.\qed

\begin{example} $\sA=DM_\gm(k)$, $\sB=DM_{\gm,\et}(k)$, $T=\alpha^s$, the change of topology functor. The hypotheses of Lemma \ref{la4} are satisfied: for the first one see Proposition 3.1 and its proof for $\sA$, and the computation of $\Gamma_\alg^\sB(TM(C))=DM_{\gm,\et}(k)(\Z,M_\et(C))_\alg$ is the same after tensoring with $\Z[1/p]$. The second hypothesis is true by [Voe00, Prop. 3.3.2]. Thus we get:
\end{example}

\begin{prop}\label{pa1} Let $A\in DM_\gm(k)$. Suppose that a universal regular homomorphism exists for $A$, with target $G$. Let $G_\et$ be the target of the universal regular homomorphism for $\alpha^s(A)$, given by Theorem \ref{ta1} c). Then the natural morphism $G\to G_\et$ is faithfully flat.\qed
\end{prop}


\end{document}